\documentclass{amsart}

\usepackage{a4wide,latexsym,amssymb,amsthm,amsmath,enumerate}

\theoremstyle{plain}

\newtheorem{Th}{Theorem}
\newtheorem{Cor}{Corollary}
\newtheorem{Pro}{Proposition}
\newtheorem{Lem}{Lemma}
\theoremstyle{definition}
\newtheorem{Def}{Definition}
\theoremstyle{remark}
\newtheorem{Rk}{Remark}

\newcommand{\w}[1]{\widetilde{#1}}

\allowdisplaybreaks[1]

\begin{document}

\title[Paracontact metric structures on $T_1 M$] {Paracontact metric structures on \\ the unit tangent sphere bundle}

\author[G.~Calvaruso]{Giovanni Calvaruso}
\address{Universit\`a del Salento\\ Dipartimento di Matematica e Fisica ``E.~De Giorgi'' \\  Provinciale Lecce--Arnesano\\ 73100 Lecce\\
Italy} \email{giovanni.calvaruso@unisalento.it}

\author[V. Mart\'in-Molina]{Ver\'onica Mart\'in-Molina}
\address{Centro Universitario de la Defensa-I.U.M.A., Academia General Militar\\
Carretera de Huesca, s/n -- Zaragoza 50090, SPAIN.} \email{vmartin@unizar.es}
\thanks{The first author is partially supported by funds of the University of Salento and MURST (PRIN 2011).}
\thanks{The second author is partially supported by the grant MTM2011-22621 (MEC, Spain), the PAI group FQM-327 (Junta de Andaluc\'{\i}a, Spain) and the E15 Grupo Consolidado Geometr\'ia (Gobierno de Arag\'on, Spain)}

\begin{abstract}
Starting from $g$-natural pseudo-Riemannian metrics of suitable signature on the unit tangent sphere bundle $T_1 M$ of a Riemannian manifold $(M,\langle,\rangle)$, we construct a family of paracontact metric structures. We prove that this class of paracontact metric structures is invariant under $\mathcal D$-homothetic deformations, and classify paraSasakian and paracontact $(\kappa,\mu)$-spaces inside this class. We also present a way to build paracontact $(\kappa,\mu)$-spaces from corresponding contact metric structures on $T_1 M$.
\end{abstract}



\maketitle


\numberwithin{equation}{section}
\renewcommand{\theequation}{\thesection.\arabic{equation}}
%

\date{}

\section{Introduction}\label{sec-intro}

In Riemannian settings, contact structures are a natural odd-dimensional analogue to complex structures. In the same way, paracontact metric structures, introduced in \cite{kaneyuki}, are the odd-dimensional counterpart to paraHermitian structures.

The study of paracontact metric manifolds focused mainly on the special case of paraSasakian manifolds, until a systematic study of paracontact metric manifolds was undertaken in recent years. The starting point was the work by S.~Zamkovoy \cite{zamkovoy}, where the Levi-Civita connection and curvature of a paracontact metric manifold were described. Paracontact $(\kappa,\mu)$-spaces were studied in \cite{CKM}. Conformal paracontact curvature was investigated in \cite{IVZ}. The first author \cite{C1} classified three-dimensional homogeneous paracontact
metric manifolds, and paracontact metric manifolds whose characteristic vector field is harmonic were studied in \cite{CP3}.

A canonical example of contact metric manifold is given by the unit tangent sphere bundle $T_1 M$, equipped with a suitable contact metric structure, having as associated metric a Riemannian metric homothetic to the Sasaki metric on $T_1 M$, and the geodesic flow vector field as the Reeb vector field \cite{blair-book}. This fact makes it natural to look at the unit tangent sphere bundle to build examples of paracontact metric structures. First results in this direction have been obtained in \cite{C}, only for the class of paracontact $(\kappa,\mu)$-spaces, considering two suitable deformations of the standard contact metric structure on the unit tangent sphere bundle over a Riemannian manifold of constant sectional curvature.

Because of the rigidity of the Sasaki metric, $g$-natural metrics have been intensively studied in recent years, providing interesting geometric behaviours under several different points of view.
These metrics were introduced first by O.~Kowalski and M.~Sekizawa \cite{KSe}, who classified second order natural transformations of Riemannian metrics on manifolds to metrics on tangent bundles.
The metrics induced on the unit tangent sphere bundle by the corresponding $g$-natural metrics on the tangent bundle $TM$ are called \emph{$g$-natural metrics on} $T_1 M$. The Sasaki metric 
is only one possible choice inside this very large family of metrics. 

In this paper we will introduce and study \emph{$g$-natural paracontact metric structures} on $T_1 M$, proving that the unit tangent sphere bundle $T_1 M$ on a Riemannian manifold $(M,\langle,\rangle)$ carries a three-parameter family of paracontact metric structures, having a pseudo-Riemannian $g$-natural metric as associated metric. We will then investigate several aspects of the paracontact metric geometry of these structures.

The paper is organised in the following way. Some basic information on $g$-natural metrics on the tangent and the unit tangent sphere bundle are provided in Section~\ref{sec-prelim}, paying particular attention to nondegeneracy and signature of these metrics. In Section~\ref{sec-para}, after reporting the needed definitions and properties about paracontact metric geometry, we will describe $g$-natural paracontact metric structures, prove their invariance by $\mathcal D$-homothetic deformations and classify paraSasakian structures and those whose tensor $\widetilde h$ satisfies $\widetilde h ^2 =0 \neq \widetilde h$. The latter case does not have any contact Riemannian analogue, due to the diagonalisability of $\widetilde h$ in the Riemannian case.

Sections~\ref{sec-kappamu} and \ref{sec-deform} will be devoted to paracontact $(\kappa,\mu)$-structures. In particular, in Section~4 we will characterize $g$-natural paracontact $(\kappa,\mu)$-spaces by means of properties on the base manifold. In Section~5 we will show that applying the deformations introduced in \cite{C} to $g$-natural contact $(\kappa,\mu)$-spaces gives us paracontact $(\kappa,\mu)$-spaces which are again $g$-natural. We end this paper with Section~\ref{sec-final}, where we consider homogeneity and harmonicity properties of $g$-natural paracontact metric structures.

\section{Preliminaries}\label{sec-prelim}

In this section we will include some needed information on $g$-natural metrics on the tangent bundle and unit tangent sphere bundle. Special  attention will be paid to the signature of these metrics.

\subsection{$g$-natural metrics on the tangent bundle}

Let $(M,\langle ,\rangle)$ be an $(n+1)$-dimensional Riemannian manifold (with $n\geq 1$) and denote by $\nabla$ its Levi-Civita connection. Then, the tangent space $(TM)_{(x,u)}$ of the tangent bundle $TM$ at a point $(x,u)$   splits as
\[
  (TM)_{(x,u)}= \mathcal{H} _{(x,u)}\oplus \mathcal{V} _{(x,u)},
\]
where $\mathcal H$ and $\mathcal V$ are the horizontal and vertical spaces with respect to $\nabla$.

Indeed, given a vector $X\in M_x$, there exists a unique vector $X^h  \in
\mathcal{H}_{(x,u)}$ (the \emph{horizontal lift} of $X$ to $(x,u)\in TM$),
such that $\pi_* X^h =X$, where $\pi:TM \rightarrow M$ is the
natural projection. The \emph{vertical lift} of a vector $X\in M_x$
to $(x,u)\in TM$ is a vector $X^v  \in \mathcal{V}_{(x,u)}$ such that $X^v
(df) =Xf$, for all functions $f$ on $M$. Here we consider $1$-forms
$df$ on $M$ as functions on $TM$ (i.e., $(df)(x,u)=uf$).

The map $X \to X^h$ is an isomorphism between the vector spaces $M_x$ and
$\mathcal{H}_{(x,u)}$. Similarly, the map $X \to X^v$ is an isomorphism
between $M_x$ and $\mathcal{V}_{(x,u)}$. Each tangent vector $Z \in
(TM)_{(x,u)}$ can be written in the form $ Z =X^h + Y^v$,
where $X,Y \in M_x$ are uniquely determined vectors. With respect to local coordinates $\{\partial/\partial x_{i}\}$ on $M$, the \emph{geodesic flow vector field} on $TM$ is uniquely determined by $u^h _{(x,u)}=\sum _i u^i \left(\partial/\partial x_{i}\right)^h _{(x,u)}$, for any point $x\in M$ and $u \in TM_x$.

We refer to the paper \cite{AS2} for a description of the class of $g$-natural metrics
 on the tangent bundle of a Riemannian manifold $(M,\langle, \rangle)$. In particular, we report
 the following characterisation.

\begin{Pro}[\cite{AS2}]\label{$g$-nat}
Let $(M,\langle, \rangle)$ be a Riemannian manifold and $G$ be a $g$-natural metric
on $TM$. Then there are six smooth functions $\alpha_i$,
$\beta_i:\mathbb{R}^+ \rightarrow \mathbb{R}$, $i=1,2,3$, such that
for every $u$, $X$, $Y \in M_x$, we have
\begin{equation}\label{exp-$g$-nat}
 \left\lbrace
\begin{array}{l}
G_{(x,u)}(X^h,Y^h) =  (\alpha_1+ \alpha_3)(r^2) \langle X,Y\rangle
                     + (\beta_1+ \beta_3)(r^2) \langle X_1,u \rangle \langle X_2,u \rangle,\\
G_{(x,u)}(X^h,Y^v) = G_{(x,u)}(X^v,Y^h)= \alpha_2 (r^2) \langle X,Y \rangle
                       +  \beta_2 (r^2) \langle X,u \rangle  \langle Y,u \rangle, \\
G_{(x,u)}(X^v,Y^v) =  \alpha_1 (r^2) \langle X,Y \rangle
                       + \beta_1 (r^2) \langle X,u \rangle \langle Y,u \rangle,
\end{array}
\arraycolsep5pt \right.
\end{equation}
where $r^2 =\langle u,u \rangle$. 
\end{Pro}

\begin{Rk}
Throughout the paper, we will use the following notation:
\[
\phi_i(t) =\alpha_i(t) +t \beta_i(t), \quad
\alpha(t) = \alpha_1(t) (\alpha_1+\alpha_3)(t) - \alpha_2^2(t), \quad
\phi(t) = \phi_1(t) (\phi_1 +\phi_3)(t) -\phi_2^2(t),
\]
for all $t \in \mathbb{R}^+$.
Unless otherwise stated, all real functions $\alpha_i$,
$\beta_i$, $\phi_i$, $\alpha$ and $\phi$ and their derivatives are
evaluated at $r ^2:=\langle u,u \rangle$.
\end{Rk}

In the literature, there are some well-known examples of Riemannian metrics on the
tangent bundle  which are special cases of Riemannian
$g$-natural metrics. In particular, the \emph{Sasaki metric} $g_S$ is obtained for
\[
\alpha _1 (t)=1, \quad \alpha _2 (t)= \alpha _3 (t)= \beta _1 (t) =\beta _2 (t)= \beta _3 (t)=0
\]
and we get the \emph{Cheeger-Gromoll metric} $g_{GC}$ for
\[
\alpha _2 (t)= \beta _2 (t)=0, \quad \alpha _1 (t)= \beta _1 (t)= -\beta _3 (t)= \frac{1}{1+t}, \quad \alpha _3 (t)= \frac{t}{1+t}.
\]

Since $\alpha _2=\beta _2 =0 $, it follows from \eqref{exp-$g$-nat}
that both the Sasaki and the Cheeger-Gromoll metrics are examples of Riemannian $g$-natural
metrics on $TM$ for which the horizontal and vertical distributions are
mutually orthogonal.

We will now investigate the signature  of a $g$-natural metric $G$ on $TM$. In particular, we will give the necessary and sufficient conditions for $G$ to be Riemannian.

Let $\{ e_0=\frac{1}{\langle u,u \rangle} u,e_1,\ldots,e_n \}$ be an orthonormal basis at $x \in M$ for $(M,\langle,\rangle )$. If we define $X_i=e_i^h$, $Y_i=e_i^v$, for $i=0,\ldots,n$, then we have that $G(X_i,X_j)=G(X_i,Y_j)=G(Y_i,Y_j)=0$, when $i \neq j$.
Therefore, the matrix of $G$ with respect to the basis $\{ X_0,Y_0,\ldots,X_n,Y_n\}$ at a point $(x,u)$ is block diagonal:
\[
G=
\begin{pmatrix}
\phi_1+\phi_3 & \phi_2  & 0                & 0       &\ldots &0                 &0        \\
\phi_2        & \phi_1  & 0                & 0       &\ldots &0                 &0        \\
0             & 0       &\alpha_1+\alpha_3 &\alpha_2 &       &                  &         \\
0             & 0       &\alpha_2          &\alpha_1 &       &\vdots            &\vdots   \\
\vdots        & \vdots  &                  &         &\ddots &0                 &0        \\
0             & 0       &   \ldots         & \ldots  &       &\alpha_1+\alpha_3 &\alpha_2 \\
0             & 0       &   \ldots         & \ldots  &       &\alpha_2          &\alpha_1
\end{pmatrix}.
\]
Consequently, it is easily seen that the determinant of $G$ is given by $\phi \cdot \alpha^n$ and its eigenvalues are $\phi_1+\phi_3 \pm \sqrt{\phi_3^2+4\phi_2^2}$ (each of them once) and $2\alpha_1+\alpha_3 \pm \sqrt{\alpha_3^2+4\alpha_2^2}$  (each of them $n$ times).

Thus, $G$ is non-degenerate if and only if $\alpha \phi \neq 0$. Moreover, $G$ is Riemannian if and only if $\alpha \phi \neq 0$ and $\phi_1+\phi_3 \pm \sqrt{\phi_3^2+4\phi_2^2}, \  \alpha_1+\alpha_3 \pm \sqrt{\alpha_3^2+4\alpha_2^2} >0$. Taking into account the notation introduced above, this is equivalent to
\begin{equation} \label{cond-Riem}
\alpha_1(t)>0, \quad \phi_1(t)>0, \quad \alpha(t)>0, \quad \phi(t)>0,
\end{equation}
for all $t \in \mathbb{R}^+$.

\subsection{$g$-natural metrics on $T_1 M$}

The \emph{tangent sphere bundle of radius $r>0$} over a Riemannian manifold $(M,\langle,\rangle )$ is the hypersurface $T_r
M= \{(x,u) \in TM \ | \ \langle u,u \rangle=r^2\}$. The tangent space at a point $(x,u) \in T_r M$ is given by
\[
(T_r M)_{(x,u)}=\{X^h +Y^v  \ / \ X \in M_x,
Y \in \{u\}^\perp \subset M_x\}.
\]
When $r=1$, $T_1 M$ is called \emph{the unit tangent (sphere) bundle}.

By a \emph{$g$-natural metric} on $T_r M$ we mean any metric $\widetilde{G}$, induced on $T_r M$ by a $g$-natural metric $G$ on $TM$. It follows from \eqref{exp-$g$-nat} that $\widetilde{G}$ is completely determined by the values of four real constants, namely,

\[
a := \alpha_1(r^2), \quad b := \alpha_2(r^2), \quad c := \alpha_3(r^2), \quad d := (\beta_1+\beta_3)(r^2).
\]
At any point $(x,u) \in T_r M$, the   metric $\widetilde{G}$ on $T_r M$ is completely described by
\begin{equation}\label{g-tilde}
\left\lbrace \arraycolsep1.5pt
\begin{array}{l}
\widetilde{G}_{(x,u)}(X_1^h,X_2^h)  = (a+c) \langle X_1,X_2 \rangle +d \langle X_1,u \rangle \langle X_2,u \rangle, \\
\widetilde{G}_{(x,u)}(X_1^h,Y_1^v) =\widetilde{G}_{(x,u)}(Y_1^v,X_1^h ) = b \langle X_1,Y_1 \rangle,\\
\widetilde{G}_{(x,u)}(Y_1^v,Y_2^v) = a \langle Y_1,Y_2 \rangle,
\end{array}
\right.
\end{equation}
for all $X_i, Y_i \in M_x$, $i=1,2$, with $Y_i$ orthogonal to $u$.


We will now study the signature of a $g$-natural metric  $\widetilde{G}$ on $T_r M$. Proceeding as for the metric $G$ on the tangent bundle $TM$, we start from an orthonormal basis $\{ e_0=u,e_1,\ldots,e_n \}$ for $(M,\langle,\rangle )$ on $x\in M$. Defining $X_0=e_0^h=u^h$ and $X_i=e_i^h$, $Y_i=e_i^v$, for $i=1,\ldots,n$, we have that $\widetilde{G}(X_i,X_j)=\widetilde{G}(X_i,Y_j)=\widetilde{G}(Y_i,Y_j)=0$, when $i \neq j$.
Therefore, the matrix of $\widetilde{G}$ with respect to the basis $\{ X_0,X_1,Y_1,\ldots,X_n,Y_n\}$ at a point $(x,u)$ is block diagonal:
\[
\widetilde{G}=
\begin{pmatrix}
{   a+c+d r^2 }         & 0         &0        &\ldots &0       &0        \\
0             &a+c        &b        &       &0       &0         \\
0             &b          &a        &       &0       &0        \\
\vdots        &           &         &\ddots &\vdots  &\vdots    \\
0             &   \ldots  & \ldots  &       &a+c     &b        \\
0             &   \ldots  & \ldots  &       &b       &a
\end{pmatrix},
\]
the determinant of $\widetilde G$ is $(a+c+d r^2 ) \alpha^{n}$, and its eigenvalues $a+c+r^2 d$ (only once) and $2a+c \pm \sqrt{c^2+4b^2}$ (each of them $n$ times).

Hence, $\widetilde G$ is Riemannian  if and only if {  $a+c+d r^2 >0$} and  $2a+c \pm \sqrt{c^2+4b^2} >0$, which is easily seen to be equivalent to
\begin{equation}\label{cond-Riem-T1M}
a > 0,\quad  a +c +d r^2 > 0, \quad \alpha=a(a +c)-b^2 > 0.
\end{equation}

\begin{Rk}\label{remark}
A $g$-natural metric $\widetilde{G}$ on $T_1 M$ is Riemannian if and only if \eqref{cond-Riem-T1M} holds, but this does not mean that the metric $\widetilde{G}$ is necessarily induced by a \emph{Riemannian} $g$-natural metric $G$ on $TM$. In fact, $G$ is Riemannian only under the extra condition $\phi=a(a+c+r^2 d)-b^2>0$, which is not necessary for $\widetilde{G}$ (see also \cite{AK2}). More precisely,
a Riemannian $g$-natural metric $\widetilde G$ on $T_r M$ is induced by:
\begin{itemize}
\item a Riemannian $g$-natural metric on $TM$ if and only if $a(a +c +dr^2) >b^2 $;
\item a degenerate $g$-natural metric of signature $(2n+1, 0,1)$ on $TM$ if and only if $a(a +c +dr^2) =b^2$;
\item  a pseudo-Riemannian $g$-natural metric of signature $(2n+1, 1,0)$ on $TM$ if and only if $a(a +c +dr^2) <b^2$.
\end{itemize}
\end{Rk}

Clearly, other signatures are also allowed for $g$-natural metrics on $T_r M$. In particular, if we require the space $\{u\}^\perp$ to be of neutral signature $(n,n)$ with respect to a (non-degenerate) metric $\widetilde{G}$, then we must have $(a+c+dr^2) \alpha^{n}\neq 0$, $2a+c + \sqrt{c^2+4b^2} >0$ and $2a+c - \sqrt{c^2+4b^2} <0$, {  and these}  conditions are equivalent to
\begin{equation}\label{cond-neutral-T1M}
a +c +d r^2\neq 0,   \quad \alpha=a(a +c) -b^2 < 0,
\end{equation}
where the sign of $a +c +d r^2$ will depend on the casual character of $u^h$.

In order to construct a paracontact metric structure with an associated $g$-natural metric on the unit tangent sphere bundle $T_1 M$, we will require the vector $u^h$ to be spacelike and the space $\{u\}^\perp$ to be of neutral signature, that is, by \eqref{cond-neutral-T1M},
\[
a+c+d>0  \quad \text{ and } \quad \alpha<0.
\]
Notice that, contrarily to the Riemannian case described by conditions \eqref{cond-Riem-T1M},  the above conditions do not give any restriction on the value of $a$. Indeed, even the case $a=0$ is possible, simply requiring that $\alpha=-b^2<0$, that is, $b \neq 0$.

Moreover, analogously to Remark~\ref{remark}, the above conditions do not yield any restriction over $\phi$.  
On the other hand, when we can reduce to the case $\phi >0$ (see Remark~\ref{phiprime} below), we can make use of the formulas for the Levi-Civita connection and curvature of $\widetilde G$ already obtained in \cite{AC1}, \cite{AC2}, while in the case $\phi<0$  some signs would necessarily change, and the case $\phi=0$ would need a completely different treatment, as it corresponds to a degenerate metric $G$ on $TM$.

Whenever $\phi \neq 0$, the vector field on $TM$ defined by
\[
N^G_{(x,u)} =\frac{1}{\sqrt{|(a+c+d) \phi| }}\, [-b u^h +(a+c+d) u^v],
\]
for all $(x,u) \in TM$, is unit normal to $T_1 M$ (either spacelike or timelike, depending on the sign of $\phi$), at any point of $T_1 M$.

With respect to $G$, the \lq\lq tangential lift\rq\rq \ $X^{t_G}$ of a vector field $X \in M_x$  to $(x,u) \in T_1 M$ is the  tangential projection of the vertical lift of $X$ to $(x,u)$ with respect to $N^G$:
\[
 X^{t_G} = X^v -\frac{\phi}{|\phi|} G_{(x,u)}(X^v,N^G_{(x,u)})\; N^G_{(x,u)}
   = X^v - \sqrt{\frac{|\phi|}{|a+c+d|}} \langle X,u \rangle \, N^G_{(x,u)}.
\]
If $X \in M_x$ is orthogonal to $u$, then $X^{t_G} = X^v$. Note that
if $b=0$, then $X^{t_G}$ coincides with the classical tangential
lift $X^t$ defined for the  case of the Sasaki metric. In the
general case,
\begin{equation*}\label{tang-tang}
X^{t_G}=X^t +\frac{b}{a+c+d} \langle X,u \rangle u^h.
\end{equation*}

The tangential lift to $(x,u) \in T_1 M$ of the vector $u$ is given
by $u^{t_G}=\frac{b}{a+c+d}\, u^h$, so $u^{t_G}$ is a horizontal
vector. Hence, we can write the tangent space to $T_1M$ at a point $(x,u)$ as
\begin{equation} \label{conv}
(T_1 M)_{(x,u)} =\{X^h +Y^{t_G} /X \in M_x,
 Y \in \{u\}^\perp \subset M_x\}.
\end{equation}
For this reason, the operation of tangential lift from $M_x$ to a
point $(x,u) \in T_1 M$ will be always applied only to vectors of
$M_x$ which are orthogonal to $u$.

Hence, an arbitrary $g$-natural metric $\widetilde{G}$ on $T_1 M$, induced by a $g$-natural metric $G$ on $TM$ with $\phi \neq 0$, is completely determined by
\begin{equation}\label{g-tilde2}
\left\lbrace \arraycolsep1.5pt
\begin{array}{l}
\widetilde{G}_{(x,u)}(X_1^h,X_2^h)  = (a+c) \langle X_1,X_2 \rangle +d \langle X_1,u \rangle \langle X_2,u \rangle, \\
\widetilde{G}_{(x,u)}(X_1^h,Y_1^{t_G}) =\widetilde{G}_{(x,u)}(Y_1^{t_G},X_1^h ) = b \langle X_1,Y_1 \rangle,\\
\widetilde{G}_{(x,u)}(Y_1^{t_G},Y_2^{t_G}) = a \langle Y_1,Y_2 \rangle,
\end{array}
\right.
\end{equation}
for all $X_i, Y_i \in M_x$, $i=1,2$, with $Y_i$ orthogonal to $u$.

Notice that the horizontal and tangential distributions  are $\widetilde G$-orthogonal to one another if and only if $b=0$. Metrics on $T_1 M$ belonging to this special subclass have been called \emph{of Kaluza-Klein type} \cite{CP1}, and have been recently used to investigate several interesting geometric behaviours \cite{CP}-\cite{CP2}. Up to our knowledge, pseudo-Riemannian $g$-natural metrics were only considered in \cite{CP2} in the context of metrics of Kaluza-Klein type, and the above discussion is the first attempt to start a systematic investigation of pseudo-Riemannian $g$-natural metrics.

The Levi-Civita connection $\widetilde \nabla$ and curvature tensor $\widetilde R$ of $(T_1 M, \widetilde G)$ can be computed using the fact that $(T_1 M,\widetilde{G})$ is a hypersurface of $(TM,G)$. When $\phi >0$, this leads to the following formulas, already obtained (again, using implicitly the assumption $\phi >0$) in \cite{AC1} and \cite{AC2} for the Riemannian case.  Throughout the paper, the curvature tensor $R$ is taken with the sign convention $R(X,Y)=[\nabla_X,\nabla_Y]\,-\,\nabla_{[X,Y]}$.

\begin{Pro}[\cite{AC1}]\label{lcc-T1M}
The Levi-Civita connection $\widetilde{\nabla}$ associated to
$\widetilde{G}$ is given, at $(x,u) \in T_1 M$, by \arraycolsep2pt
\begin{align*}
(\widetilde{\nabla}_{X^h}Y^h)_{(x,u)} & =   \left\{(\nabla_X Y)_x
                       -\frac{ab}{2\alpha} [R(X_x,u)Y_x
                       + R(Y_x,u)X_x] \right.
                     +\frac{bd}{2\alpha} [\langle X,u \rangle Y_x
                        + \langle Y,u \rangle X_x] \\
                   & \qquad  +\frac{b}{(a+c+d)\alpha} [(ad+b^2)\,
                       \langle R(X_x,u)Y_x,u\rangle
                     \left.-\frac{}{}d(a+c+d)\,
                      \langle X,u \rangle\langle Y,u \rangle]u\right\}^h \\
                   &  + \left\{\frac{b^2}{\alpha} \,R(X_x,u)Y_x
                       -\frac{a(a+c)}{2\alpha} \,R(X_x,Y_x)u \right.
                     -\frac{(a+c)d}{2\alpha} \,
                       [\langle Y,u \rangle\,X_x +\langle X,u \rangle\,Y_x] \\
                   & \qquad  +\frac{1}{\alpha} [-b^2\,
                       \langle R(X_x,u)Y_x,u\rangle
                     \left.+\frac{}{}d(a+c) \,
                      \langle Y,u \rangle\langle X,u \rangle]u \right\}^{t_G},\\
(\widetilde{\nabla}_{X^h}Y^{t_G})_{(x,u)} & =
                  \left\{-\frac{a^2}{2\alpha} \,R(Y_x,u)X_x
                  \right.
                   +\frac{ad}{2\alpha}\,\langle X,u \rangle\,Y_x\\
                 & \left. \qquad
                   +\frac{1}{2(a+c+d)\alpha}\,[a(ad+b^2)\,\langle R(X_x,u)Y_x,u \rangle+d \alpha\,\langle X,Y \rangle]u\right\}^h \\
                 &  +\left\{(\nabla_X Y)_x
                     +\frac{ab}{2\alpha} \,R(Y_x,u)X_x
                   -\frac{bd}{2\alpha}\,\langle X,u \rangle\,Y_x
                   -\frac{ab}{2\alpha}\,\langle R(X_x,u)Y_x,u\rangle u \right\}^{t_G},\\
(\widetilde{\nabla}_{X^{t_G}}Y^h)_{(x,u)} & =
                  \left\{-\frac{a^2}{2\alpha} \,R(X_x,u)Y_x
                    +\frac{ad}{2\alpha}\,\langle Y,u \rangle\,X_x \right.\\
                 & \qquad \left. +\frac{1}{2(a+c+d)\alpha}\,[a(ad
                      +b^2)\,\langle R(X_x,u)Y_x,u \rangle +d \alpha\,\langle X,Y \rangle ]u\right\}^h \\
                 &  +\left\{ \frac{ab}{2\alpha} \,R(X_x,u)Y_x
                     -\frac{bd}{2\alpha}\,\langle Y,u \rangle\,X_x
                    -\frac{ab}{2\alpha}\,\langle R(X_x,u)Y_x,u \rangle u\right\}^{t_G},\\
(\widetilde{\nabla}_{X^{t_G}}Y^{t_G})_{(x,u)} & =0,
\end{align*}
for all arbitrary vectors $X,Y \in M_x$ satisfying the convention in \eqref{conv}.
\end{Pro}

\begin{Pro}[\cite{AC2}]\label{riem-curv-T1M}
Let $(M,\langle,\rangle )$ be a Riemannian manifold, $G$ a $g$-natural metric on $TM$ and $\widetilde{G}$ the $g$-natural metric on $T_1 M$ induced by $G$, which is determined by \eqref{g-tilde2} with $a,b,c,d$ constants. Then the Riemannian curvature tensor at $(x,u) \in T_1 M$ is given by \arraycolsep0pt
\begin{align*}
& (i) \widetilde{R} (  X^h , Y^h  ) Z^h \\
      & = \left\{ R(X,Y)Z + \frac{ab}{2\alpha} \,[2(\nabla_u R)(X,Y)Z
          -(\nabla_Z R)(X,Y)u]\right.  \\
      &+\frac{a^2}{4\alpha} \,[R(R(Y,Z)u,u)X -R(R(X,Z)u,u)Y-2R(R(X,Y)u,u)Z]\\
      & +\frac{a^2b^2}{4\alpha^2} \,[R(X,u)R(Y,u)Z -R(Y,u)R(X,u)Z+R(X,u)R(Z,u)Y -R(Y,u)R(Z,u)X]\\
      & + \frac{ad(\alpha-b^2)}{4\alpha^2}\,[\langle Z,u \rangle R(X,Y)u +\langle Y,u \rangle R(X,u)Z -\langle X,u \rangle R(Y,u)Z ]\\
      & +\frac{ab^2}{2\alpha^2} \,\left[-\frac{ad+b^2}{a+c+d}\,  \langle R(Y,u)Z,u \rangle +d\, \langle Y,u \rangle \langle Z,u \rangle\right]R_u X\\
      & -\frac{ab^2}{2\alpha^2} \,\left[-\frac{ad+b^2}{a+c+d}\,
        \langle R(X,u)Z,u\rangle +d\, \langle X,u \rangle\langle Z,u \rangle\right]R_u Y \\
      & +\frac{d}{4\alpha} \,\left[-\frac{2b^2}{a+c+d}\,\langle R(Y,u)Z,u \rangle
        +d\,\langle Y,u \rangle \langle Z,u \rangle\right]X \\
      & -\frac{d}{4\alpha} \,\left[-\frac{2b^2}{a+c+d}\,\langle R(X,u)Z,u \rangle
        +d\,\langle X,u \rangle\langle Z,u \rangle\right]Y\\
       &   +\frac{d}{4\alpha(a+c+d)} \left\{-4ab \langle (\nabla_u R)(X,Y)Z,u \rangle \right.\\
       &   +a^2\,[ \langle R(Y,Z)u,R(X,u)u \rangle
           -\langle R(X,Z)u,R(Y,u)u\rangle -2 \langle R(X,Y)u,R(Z,u)u \rangle] \\
       &   +\frac{a^2b^2}{\alpha}\,[\langle R(Y,u)Z+R(Z,u)Y,R(X,u)u\rangle
          -\langle R(X,u)Z+R(Z,u)X,R(Y,u)u\rangle] \\
       &   -\left[\frac{ad(b^2 -\alpha)}{\alpha}
          +\frac{2b^2 d(\phi+2b^2)}{\phi(a+c+d)}
          +\frac{4b^2 \alpha}{\phi} \right]
          \,[\langle X,u \rangle\langle R(Y,u)Z,u \rangle
         -\langle Y,u \rangle\langle R(X,u)Z,u \rangle]\\
         & - 3a(a+c)\, \langle R(X,Y)Z,u \rangle
          +(a+c)d\, [\langle X,u \rangle  \langle Y,Z \rangle
          -\langle Y,u \rangle \langle X,Z \rangle]\left. \left. \right\}u \right\} ^h\\
           &   +\left\{ -\frac{b^2}{\alpha} \,(\nabla_u R)(X,Y)Z
          +\frac{a(a+c)}{2\alpha}(\nabla_Z R)(X,Y)u\right.\\
          &  -\frac{ab}{4\alpha} \,[R(R(Y,Z)u,u)X -R(R(X,Z)u,u)Y -2R(R(X,Y)u,u)Z\\
            & -R(X,R(Y,u)Z)u -R(X,R(Z,u)Y)u +R(Y,R(X,u)Z)u +R(Y,R(Z,u)X)u]\\
            &   -\frac{ab^3}{4\alpha^2} \,[R(X,u)R(Y,u)Z -R(Y,u)R(X,u)Z +R(X,u)R(Z,u)Y -R(Y,u)R(Z,u)X]\\
            & -\frac{bd(3\alpha-b^2)}{4\alpha^2}\,[\langle Z,u \rangle R(X,Y)u +\langle Y,u \rangle R(X,u)Z -\langle X,u \rangle R(Y,u)Z ]\\
            & +\frac{b(b^2-\alpha)}{2\alpha^2} \, \left[\frac{ad+b^2}{a+c+d}\,\langle R(Y,u)Z,u \rangle -d \, \langle Y,u \rangle\langle Z,u \rangle\right]R_u X \\
            & -\frac{b(b^2-\alpha)}{2\alpha^2} \,
        \left[\frac{ad+b^2}{a+c+d}\,\langle R(X,u)Z,u \rangle -d \, \langle X,u \rangle\langle Z,u \rangle\right]R_u Y \\
      & \left. +\frac{(a+c)bd}{2\alpha(a+c+d)} \,[\langle R(Y,u)Z,u \rangle X
        -\langle R(X,u)Z,u \rangle Y ] \right\} ^{t_G},
\end{align*}
\begin{align*}
      & (ii) \widetilde{R} (  X^h,  Y^{t_G}  ) Z^h \\
      & = \left\{ -\frac{a^2}{2\alpha} \,(\nabla_X R)(Y,u)Z
          +\frac{ab}{2\alpha}[R(X,Y)Z +R(Z,Y)X]\right. \\
      &   +\frac{a^3b}{4\alpha^2} \,[R(X,u)R(Y,u)Z
        -R(Y,u)R(X,u)Z -R(Y,u)R(Z,u)X] \\
      &+ \frac{a^2bd}{4\alpha^2}\,[\langle X,u \rangle R(Y,u)Z -\langle Z,u \rangle R(X,Y)u]\\
      & -\frac{ab}{4\alpha^2(a+c+d)} \,[a(ad+b^2)\,\langle R(Y,u)Z,u \rangle
        +\alpha d\, \langle Y,Z \rangle  ]R_u X \\
      & +\frac{a^2b}{2\alpha^2} \,\left[\frac{ad+b^2}{a+c+d}\,\langle R(X,u)Z,u \rangle
        -d\, \langle X,u \rangle\langle Z,u \rangle\right]R_u Y \\
      & -\frac{bd}{4\alpha(a+c+d)} \,[a\,\langle R(Y,u)Z,u \rangle
        +(2(a+c)+d)\,  \langle Y,Z \rangle ]X \\
      & +\frac{b}{\alpha} \,\left[-\frac{ad+b^2}{2(a+c+d)}\,\langle R(X,u)Z,u \rangle
        +d\, \langle X,u \rangle\langle Z,u \rangle\right]Y
       -\frac{bd}{2\alpha}\,\langle X,Y \rangle Z\\
      &\left.+\frac{d}{4\alpha(a+c+d)} \right\{2a^2\,\langle (\nabla_X
        R)(Y,u)Z,u \rangle +\frac{a^3b}{\alpha}\,[\langle R(Y,u)Z,R(X,u)u \rangle  \\
               &            -\langle R(X,u)Z+R(Z,u)X,R(Y,u)u \rangle ]  +ab\,\left[-\frac{\alpha+\phi}{\alpha}
        +\frac{d}{a+c+d}\right]\, \langle X,u \rangle\langle R(Y,u)Z,u \rangle \\
      &  - 2ab\, [2 \langle R(X,Y)Z,u\rangle +\langle R(Z,Y)X,u\rangle ] \\
      & \left.\left. +bd \,\left[\left(3-\frac{d}{a+c+d}\right)\,
        \langle X,u \rangle  \langle Y,Z \rangle +2\, \langle Z,u \rangle \langle X,Y \rangle \right]\right\}u \right\} ^h \\
      &   +\left\{\frac{ab}{2\alpha} \,(\nabla_X R)(Y,u)Z
         +\frac{a^2}{4\alpha}\, R(X,R(Y,u)Z)u \right.\\
      &     -\frac{a^2b^2}{4\alpha^2} \,[R(X,u)R(Y,u)Z  -R(Y,u)R(X,u)Z -R(Y,u)R(Z,u)X]\\
      &  -\frac{b^2}{\alpha}\, R(X,Y)Z +\frac{a(a+c)}{2\alpha}R(X,Z)Y
         + \frac{ad(\alpha-b^2)}{4\alpha^2}\,[\langle X,u \rangle R(Y,u)Z
         -\langle Z,u \rangle R(X,Y)u] \\
      & -\frac{\alpha-b^2}{4\alpha^2(a+c+d)} \,[a(ad+b^2)\,\langle R(Y,u)Z,u \rangle
        +\alpha d\, \langle Y,Z \rangle  ]R_u X \\
      & +\frac{ab^2}{2\alpha^2} \,\left[-\frac{ad+b^2}{a+c+d}\,
        \langle R(X,u)Z,u \rangle +d\, \langle X,u \rangle\langle Z,u \rangle\right]R_u Y \\
      & +\frac{(a+c)d}{4\alpha(a+c+d)} \,[a\,\langle R(Y,u)Z,u \rangle
        +(2(a+c)+d)\,  \langle Y,Z \rangle ]X \\
      &  +\frac{1}{4\alpha} \,
        \left[2b^2\left(2-\frac{d}{a+c+d}\right)\,\langle R(X,u)Z,u \rangle
         -d(4(a+c)+d)\, \langle X,u \rangle\langle Z,u \rangle\right]Y   \\
      &\left. +\frac{(a+c)d}{2\alpha}\, \langle X,Y \rangle Z  \right\}^{t_G},
\end{align*}
\begin{align*}
 & (iii) \widetilde{R} (  X^{t_G} , Y^{t_G}  ) Z^{t_G} =
\frac{1}{2\alpha(a+c+d)} \left[ \left\{a^2b\frac{}{}\,
[ \langle Y,Z \rangle R_u X -\langle X,Z \rangle R_u Y]\right.\right. \\
  & \left. \frac{}{}-b(\alpha+\phi)[ \langle Y,Z \rangle X-\langle X,Z \rangle Y] \right\}^h  +\left\{-ab^2\frac{}{}\,[ \langle Y,Z \rangle R_u X -\langle X,Z \rangle R_u Y]\right. \\
  & \left. \left. +[(a+c)(\alpha+\phi) +\alpha d]\,
    [ \langle Y,Z \rangle X-\langle X,Z \rangle Y] \right\}^{t_G}\frac{}{}\right],
\end{align*}
for all arbitrary vectors $X,Y,Z \in M_x$ satisfying the convention \eqref{conv}, where
$R_u X = R(X,u)u$ denotes the \emph{Jacobi operator} associated to
$u$.
\end{Pro}

\section{Paracontact $g$-natural  metric structures}\label{sec-para}

An \emph{almost paracontact structure} $(\varphi,\xi,\eta)$ (as defined in \cite{kaneyuki} and later investigated in \cite{zamkovoy}) on a
$(2n+1)$-dimensional smooth manifold is given by a
$(1,1)$-tensor field $ \varphi$, a vector field $\xi$ and a $1$-form $\eta$, satisfying  
\begin{enumerate}[(i)]
  \item $\eta(\xi)=1, \quad  \varphi^2=I-\eta\otimes\xi$, and
  \vspace{1mm}\item the eigendistributions ${\mathcal D}^+$ and ${\mathcal D}^-$ of $ \varphi$ corresponding to the eigenvalues $1$ and $-1$ respectively, have equal dimension $n$.
\end{enumerate}
A pseudo-Riemannian metric $g$ is said to be \emph{compatible} with the paracontact structure if
\begin{equation}\label{compatibile}
  g( \varphi X, \varphi Y)=-  g(X,Y)+\eta(X)\eta(Y),
\end{equation}
for all $X,Y$ vector fields on $M$.

It follows from the definition that $ \varphi\xi=0$, $\eta\circ  \varphi=0$ and $\varphi$ has rank $2n$. Moreover, a compatible metric $g$ is necessarily of signature $(n+1,n)$, with $\xi$ unit and spacelike  and the distribution $\{\xi\}^\perp$ of neutral signature $(n,n)$. Finally, equation \eqref{compatibile} also yields that $\eta= {g}(\cdot,\xi)$ and $ {g}(\cdot, \varphi\cdot)=- {g}( \varphi\cdot,\cdot)$.

The \emph{fundamental $2$-form} of the almost paracontact metric manifold is defined by
$ \Phi(X,Y)= {g}(X, \varphi Y)$. If $d\eta= \Phi$, then $\eta$ is a contact form, $g$ is said to be an \emph{associated metric} and $(M, \varphi,\xi,\eta,  g)$ is called a \emph{paracontact metric manifold}.

A paracontact metric structure $(\varphi,\xi,\eta,g)$ is said to be \emph{K-paracontact} if $\xi$ is a Killing vector field. This is equivalent to requiring that $h=0$, where $h=\frac{1}{2} L_{\xi} \varphi$ and $L$ is the usual Lie derivative.
On a paracontact metric manifold \cite{zamkovoy}, one has that $h$ is self-adjoint, $h (\xi)= \text{ tr} (h)=0$
and
\begin{equation}\label{h-properties}
\nabla {\xi}=-\varphi +\varphi {h}, \quad {h} \varphi=-\varphi h, \quad \eta \circ {h}=0.
\end{equation}
Like in contact metric geometry, the tensor $h$ is essential in describing the geometric properties of a paracontact metric structure.

A paracontact metric manifold $(M,\varphi,\xi,\eta,g)$ is called \emph{paraSasakian} when it is normal, that is, satisfies the integrability condition
\[
N_{\varphi}:=[{\varphi},{\varphi}]-2d \eta \otimes \xi=0.
\]
This is equivalent to
\begin{equation}\label{parasasakian}
(\nabla_{Z} \varphi)W=-g(Z,W)\xi+\eta(W)Z,
\end{equation}
for all vector fields $Z,W$ tangent to $M$. Every paraSasakian manifold is K-paracontact. The converse holds in dimension $3$ but in general fails in higher dimension.
For these and further results on paracontact metric structures, we refer to \cite{zamkovoy}.

We will now see how to define a paracontact structure on $T_1 M$, having $g$-natural metrics $\widetilde{G}$ (of suitable signature) as associated metrics.

\begin{Th}\label{th-paracontact}
Let $\widetilde G$ denote a pseudo-Riemannian $g$-natural metric on the unit tangent sphere bundle $T_1 M$, described as in \eqref{g-tilde2}. 
%
%
Consider the paracontact structure  $(\widetilde\varphi,\widetilde\xi,\widetilde\eta)$, where we put
\begin{equation}\label{xi}
\widetilde\xi=\rho u^h,
\end{equation}
for some real constant $\rho>0$,

\begin{equation} \label{eta}
\widetilde\eta(X^h)=\frac{1}{\rho} \langle X,u \rangle, \qquad \widetilde\eta(Y^{t_G})=b\rho \langle Y,u \rangle,
\end{equation}
and
\begin{equation} \label{phi}
\left\{
\begin{aligned}
\widetilde\varphi(X^h)&=\frac{1}{2\rho\alpha}\left( -bX^h+(a+c)X^{t_G}+\frac{bd}{a+c+d} \langle X,u\rangle u^h \right),\\
\widetilde\varphi(Y^{t_G})&=\frac{1}{2\rho\alpha}\left( -aY^h+b Y^{t_G} +\frac{\phi}{a+c+d} \langle Y,u \rangle u^h \right),\\
\end{aligned}
\right.
\end{equation}
for all $X,Y \in M_x$.

Then, $\widetilde G$ is an associated metric for $(\widetilde\varphi,\widetilde\xi,\widetilde\eta)$ if and only if
\begin{align}\label{para-cond}
a+c+d=-4\alpha=\frac{1}{\rho^2}.
\end{align}
\end{Th}
\begin{proof}
We start by introducing $\widetilde{\xi}:=\rho u^h$, where $\rho$ is a positive constant. Then, equation~\eqref{conv} yields that the tangent space of $T_1 M$  at $(x,u)$ can be written as
\begin{equation} \label{conv2}
(T_1 M)_{(x,u)} =\text{span}(\widetilde\xi) \oplus \{X^h \ / \ X\perp u \} \oplus \{ Y^{t_G} \ / \ Y \perp u\}.
\end{equation}
As we already remarked in the previous section, requiring that $\widetilde\xi$ (collinear to $u^h$) is spacelike and that the restriction of $\widetilde G$ on ${\xi}^\perp$ is of neutral signature, by \eqref{cond-neutral-T1M} we find that the constants $a,b,c,d$ must satisfy $a+c+d>0$ and $\alpha=a(a+c)-b^2<0$. Moreover, it follows from \eqref{g-tilde2} that $\widetilde\xi$ is unit if and only if $\rho^2(a+c+d)=1$.

We now define $\widetilde{\eta}$ as the $1$-form dual to $\widetilde{\xi}$ with respect to $\widetilde{G}$, and $\widetilde{\varphi}$ by the condition $\widetilde{G}(\cdot,\widetilde{\varphi} \cdot)=(d\widetilde{\eta})(\cdot, \cdot)$. Formulas~\eqref{eta} and \eqref{phi} then follow directly from \eqref{g-tilde2}.

Next, we impose that the condition $\widetilde\varphi^2=I-\widetilde\eta\otimes\widetilde\xi$  holds for all $X^h$ and $Y^{t_G}$, with $X,Y \in M_x$ and $Y$ orthogonal to $u$, and we obtain
\begin{align*}
X^h-\langle X,u \rangle u^h&=-\frac{1}{4\rho^2 \alpha}(X^h-\langle X,u \rangle u^h),\qquad
Y^{t_G}=-\frac{1}{4\rho^2 \alpha} Y^{t_G}.
\end{align*}
Hence, $-4\rho^2\alpha=1$, which completes the proof of equation~\eqref{para-cond}. Taking into account this equation, it is now easy to check that condition~\eqref{compatibile} is satisfied.

We now prove that $\dim{\mathcal D}^+= \dim{\mathcal D}^-=n$, where ${\mathcal D}^\pm$ are the eigendistributions $\widetilde\varphi$ corresponding to the eigenvalues $\pm 1$.

An arbitrary $Z \in Ker \widetilde\eta$ (that is, orthogonal to $\widetilde\xi$) can be  written as $Z=X^h+Y^{t_G}$, with $X,Y \in M_x$ orthogonal to $u$.
Therefore, by equation~\eqref{phi} we find that $Z_1=X_1+Y_1 \in \mathcal{D}_{\widetilde\varphi}(1)$ if and only if
\[
X_1^h+Y_1^{t_G}=\left\{ -\frac{1}{2\rho\alpha}(bX_1+aY_1)\right\}^h +\left\{ -\frac{1}{2\rho\alpha}((a+c)X_1+bY_1)\right\}^{t_G}.
\]
Since the horizontal and tangential parts of a vector field are uniquely determined, the above condition yields
\begin{equation}\label{system-1}
\left\{
\begin{aligned}
&(b+2\rho\alpha)X_1 +aY_1     =0,\\
&(a+c)X_1+(b-2\rho\alpha)Y_1 =0.
\end{aligned}
\right.
\end{equation}
Notice that the two equations in \eqref{system-1} are linearly dependent because of condition $4\rho^2\alpha=-1$.

In the same way, $Z_2=X_2+Y_2 \in \mathcal{D}_{\widetilde\varphi}(-1)$ if and only if
\begin{equation}\label{system-2}
\left\{
\begin{aligned}
&(b-2\rho\alpha)X_2 +aY_2     =0,\\
&(a+c)X_2+(b+2\rho\alpha)Y_2  =0,
\end{aligned}
\right.
\end{equation}
and the two equations are again linearly dependent since $4\rho^2\alpha=-1$. We will now consider two different cases, according on whether $a=0$ or $a \neq 0$.

If $a\neq 0$, then \eqref{system-1} and \eqref{system-2} yield that
\begin{align*}
\mathcal{D}_{\widetilde\varphi}(1)= \left\{ X_1^h-\frac{b+2\rho\alpha}{a}X_1^{t_G} \ / \ X_1 \perp u \right\},\qquad
\mathcal{D}_{\widetilde\varphi}(-1)= \left\{ X_2^h-\frac{b-2\rho\alpha}{a}X_2^{t_G} \ / \ X_2 \perp u \right\}.
\end{align*}

So, considering an orthonormal basis $\{ e_0=u,e_1,\ldots,e_n\}$ on $T_xM$, we can construct  a basis on $T_{(x,u)}(T_1M)$ as $\{ e_0^h=u^h,e_1^h,e_1^{t_G},\ldots,e_n^h,e_n^{t_G}\}$. Hence,
\[
\mathcal{D}_{\widetilde\varphi}(\pm 1)=\text{span} \left( e_i^h-\frac{b \pm 2\rho\alpha}{a}e_i^{t_G} \ / \ i=1, \ldots, n \right),
\]
and both eigendistributions have dimension $n$.

If $a=0$, then  $\alpha=-b^2<0$ and so, $b\neq0$.  Systems~\eqref{system-1} and \eqref{system-2} then respectively become
\begin{equation}\label{system-a0}
\left\{
\begin{aligned}
&(1-2\rho b)X_1     =0,\\
&cX_1+b(1+2\rho b)Y_1 =0,
\end{aligned}
\right.
\quad
\text{ and}
 \quad
\left\{
\begin{aligned}
&(1+2\rho b)X_2     =0,\\
&cX_2+b(1-2\rho b)Y_2  =0.
\end{aligned}
\right.
\end{equation}
Equation \eqref{para-cond} now yields that $\frac{1}{4\rho^2}=-\alpha=b^2$. If $b=\frac{1}{2\rho}>0$, then the equations in \eqref{system-a0} reduce to $cX_1+2bY_1=0$ and $X_2=0$. So,
\begin{align*}
\mathcal{D}_{\widetilde\varphi}(1)&= \left\{ X_1^h-\frac{c}{2b}X_1^{t_G} \ / \ X_1 \perp u \right\},\qquad
\mathcal{D}_{\widetilde\varphi}(-1)= \left\{ Y_2^{t_G} \ / \ Y_2 \perp u \right\}.
\end{align*}
By a similar argument, if $b=-\frac{1}{2\rho}<0$, then we get $X_1=0$ and $cX_2+2bY_2=0$, so that
\begin{align*}
\mathcal{D}_{\widetilde\varphi}(1)&= \left\{ Y_1^{t_G} \ / \ Y_1 \perp u \right\},\qquad
\mathcal{D}_{\widetilde\varphi}(-1)= \left\{ X_2^h-\frac{c}{2b}X_2^{t_G} \ / \ X_2 \perp u \right\}.
\end{align*}
Using the basis $\{ e_0^h=u^h,e_1^h,e_1^{t_G},\ldots,e_n^h,e_n^{t_G}\}$ on $T_{(x,u)}(T_1M)$, we obtain
\begin{align*}
& \left\{ Y_2^{t_G} \ / \ Y_2 \perp u \right\}=\left\{ Y_1^{t_G} \ / \ Y_1 \perp u \right\}=\text{span} \left( e_i^{t_G} \ / \  i=1, \ldots, n \right),\\
&\left\{ X_1^h-\frac{c}{2b}X_1^{t_G} \ / \ X_1 \perp u \right\}=\left\{ X_2^h-\frac{c}{2b}X_2^{t_G} \ / \ X_2 \perp u \right\}=\text{span} \left( e_i^h-\frac{c}{2b}e_i^{t_G} \ / \ i=1, \ldots, n \right),
\end{align*}
so  $\dim \mathcal{D}_{\widetilde\varphi}(1)=\dim\mathcal{D}_{\widetilde\varphi}(-1)=n$ in both of the above cases.
\end{proof}

\begin{Rk}
(1) Since the paracontact distribution ker$\widetilde \eta$ is given by
\[
{\rm ker} \widetilde \eta = {\widetilde \xi}^\perp = \left\{X^h+Y^{t_G} \ / \ X,Y \perp u  \right\},
\]
it would suffice to write \eqref{eta} and \eqref{phi} for vectors orthogonal to $u$. However, we need the description of the extra terms involving $\langle X,u\rangle, \langle Y,u\rangle$, because $\langle Y,u \rangle=0$ at $x\in M$ does not mean that $\langle Y,u \rangle=0$ everywhere on $M$. So, $\w\nabla_Z (\langle Y,u \rangle) \neq0$ in general, and this term must be taken into account when calculating  covariant derivatives.

\smallskip
(2) Equation \eqref{para-cond} allows us to write $d$ in terms of $a,b,c$. In fact, we find
\[
d=(a+c+d)-(a+c)=-4\alpha-(a+c)
=-(a+c)(4a+1)+4b^2.
\]
Thus, the \emph{paracontact metric structures described in the above Theorem~\ref{th-paracontact} depend on three parameters $a,b,c$} (satisfying conditions \eqref{para-cond}).
\end{Rk}

\begin{Def}
A \emph{$g$-natural paracontact metric structure} on $T_1 M$ is any paracontact metric structure $(\widetilde\varphi,\widetilde\xi,\widetilde\eta, \widetilde G)$ described by equations~\eqref{g-tilde2} and
\eqref{xi}-\eqref{phi}.
\end{Def}

We will now consider $\mathcal D$-homothetic deformations of a $g$-natural paracontact metric structure. Given a paracontact metric structure $(\varphi,\xi,\eta,g)$ and a real constant $t\neq 0$, the \emph{${\mathcal D}_t$-homothetic deformation} of $(\varphi,\xi,\eta,g)$ \cite{zamkovoy} is the new paracontact metric structure $(\varphi _t,\xi _t,\eta _t,g_t)$, described by
\begin{equation}\label{Dhom}
g_t=t g +t(t-1) \eta \otimes \eta, \quad \eta_t=t \eta, \quad \varphi_t=\varphi, \quad \xi_t=\frac{1}{t}\xi.
\end{equation}
Consider now an arbitrary $g$-natural paracontact metric structure $(\widetilde\varphi,\widetilde\xi,\widetilde\eta, \widetilde G)$. If we apply a $\mathcal D$-homothetic deformation to $(\widetilde\varphi,\widetilde\xi,\widetilde\eta, \widetilde G)$, we obtain the paracontact metric structure $(\widetilde\varphi _t,\widetilde\xi _t,\widetilde\eta _t, \widetilde G_t)$, which is again $g$-natural.

In fact, it is easy to check that equations~\eqref{g-tilde2} and
\eqref{xi}-\eqref{para-cond} hold for $(\widetilde\varphi _t,\widetilde\xi _t,\widetilde\eta _t, \widetilde G_t)$, taking $\rho_t=\frac{\rho}{t}$, $a_t=ta$, $b_t=tb$, $c_ t=tc$ and $d_t=t(d+\frac{t-1}{\rho^2})$. Therefore, we proved the following result.

\begin{Th}\label{th-Dhom}
The class of $g$-natural paracontact metric structures $\left\{(\widetilde\varphi,\widetilde\xi,\widetilde\eta, \widetilde G)\right\}$ is invariant under $\mathcal D$-homothetic deformations.
\end{Th}

\begin{Rk}\label{phiprime}
Consider the $\mathcal D_t$-homothetic deformation $(\widetilde\varphi _t,\widetilde\xi _t,\widetilde\eta _t, \widetilde G_t)$ of a $g$-natural paracontact metric structure $(\widetilde\varphi ,\widetilde\xi ,\widetilde\eta , \widetilde G)$. Since $t \neq 0$ and $a_t=t a, b_t=tb,c_t=tc$, the vanishing (or not vanishing) of these coefficients are properties invariant for $\mathcal D$-homothetic deformations.  Moreover, we have
\[
\phi_t=t^2 \phi +at^2(t-1)/\rho^2.
\]
Therefore,  different behaviours occur, according on whether $a=0$ or $a \neq 0$. In fact:
\begin{itemize}
\item[(i)] If $a=0$, then $\phi = -b^2 <0$ and $\phi _t = t^2 \phi<0$ for any $t\neq0$. Hence, $\phi<0$ remains invariant for $\mathcal D$-homothetic deformations involving a (pseudo-Riemannian) $g$-natural metric $\widetilde G$ with $a=0$.
\vspace{1mm}\item[(ii)]
If $a \neq 0$, then whatever the value of $\phi$, there exists a $\mathcal D$-homothetic deformation of the paracontact metric structure, for which $\phi _t>0$. In fact, if $a>0$ (respectively, $a<0$), then $\phi _t$ goes to $+\infty$ as $t$ goes to $+\infty$ (respectively, to $-\infty$).
\end{itemize}
\end{Rk}

\medskip
We will now investigate the geometry of $g$-natural paracontact metric structures on $T_1 M$. 
We start from the classification of the paraSasakian structures, proving the following result.

\begin{Th}\label{th-parasasakian}
For any $g$-natural paracontact metric structure $(\widetilde\varphi,\widetilde\xi,\widetilde\eta,\widetilde{G})$ on $T_1 M$, constructed from a $g$-natural metric $\widetilde G$ with $a \neq 0$, the following properties are equivalent:
\begin{itemize}
\item[(i)] $(\widetilde\varphi,\widetilde\xi,\widetilde\eta,\widetilde{G})$ is paraSasakian;
\vspace{1mm}\item[(ii)] $(\widetilde\varphi,\widetilde\xi,\widetilde\eta,\widetilde{G})$ is $K$-paracontact;
\vspace{1mm}\item[(iii)]  $b=0$ and the base manifold $(M,\langle,\rangle)$ has constant  sectional curvature $\bar{c}=\frac{a+c}{a} <0$.
\end{itemize}
\end{Th}

\begin{proof} \lq\lq(i)$\Rightarrow$(ii)\rq\rq: \ It holds in general.

\smallskip
\lq\lq(ii)$\Rightarrow$(iii)\rq\rq:  Consider a $g$-natural paracontact metric manifold $(T_1M, \widetilde\varphi,\widetilde\xi,\widetilde\eta,\widetilde{G})$. Because of Remark~\ref{phiprime} and recalling that the property of being paraSasakian is invariant under $\mathcal D$-homothetic deformations \cite{zamkovoy}, without loss of generality we can assume that $\phi>0$. Then, using the description of the Levi-Civita connection  of a $g$-natural metric $\widetilde{G}$ given in Proposition~\ref{lcc-T1M}, we find
\begin{equation}\label{nabla-x-xi}
\left\{
\begin{aligned}
\widetilde\nabla_{X^h} \widetilde\xi     &=\frac{\rho}{2\alpha}  \{ bd X-ab R_u X \}^h
+ \frac{\rho}{2\alpha}  \{ -(a+c)d X+(b^2-\alpha) R_u X \}^{t_G},\\
\widetilde\nabla_{Y^{t_G}} \widetilde\xi &=\frac{\rho}{2\alpha}  \{ (ad+2\alpha)X-a^2 R_u X \}^h
+ \frac{\rho}{2\alpha}  \{ -bd X+ad R_u X \}^{t_G},
\end{aligned}
\right.
\end{equation}
for all vector fields $X,Y \in M_x$ orthogonal to $u$, where $R_uX=R(X,u)u$ is the Jacobi operator associated to $u$. If $X^h=\widetilde{\xi}$, then $\widetilde\nabla_{\widetilde\xi} \widetilde\xi=0$ (as in any paracontact metric manifold \cite{zamkovoy}).

Applying $\widetilde\varphi$ to the first equation in \eqref{h-properties}, we obtain $\widetilde\varphi (\widetilde\nabla {\widetilde\xi})=-\widetilde\varphi^2 +\widetilde\varphi^2 \widetilde{h}=-I+\widetilde\eta \otimes \widetilde\xi + \widetilde{h}$,  and so
\[
\widetilde{h}=\widetilde\varphi (\widetilde\nabla {\widetilde\xi})+I-\widetilde\eta \otimes \widetilde\xi.
\]
Thus, \eqref{nabla-x-xi} yields
\begin{equation*}
\left\{
\begin{aligned}
\widetilde{h}(X^h)&=\frac{1}{4\alpha}\left\{  -(a+c)X^h+a(R_uX)^h-2b (R_u X)^{t_G}\right\},\\
\widetilde{h}(Y^{t_G})&=\frac{1}{4\alpha} \left\{-2b Y^h+(a+c)Y^{t_G}-a (R_u Y)^{t_G}\right\},
\end{aligned}
\right.
\end{equation*}
for all $X,Y \in M_x$ orthogonal to $u$.

If $X$ is arbitrary, then $X-\eta(X)\xi$ is orthogonal to $u$. Taking into account $h(u^h)=0$ and $R_u (X-\eta(X)u)=R_u X$, we then conclude that $\widetilde{h}$ is completely determined by
\begin{equation}\label{h-explicit}
\left\{
\begin{aligned}
\widetilde{h}(X^h)
&=\frac{1}{4\alpha}\left\{ -(a+c)(X-\langle X, u\rangle u)^h+a(R_uX)^h-2b (R_u X)^{t_G}\right\},\\
\widetilde{h}(Y^{t_G})&=\frac{1}{4\alpha} \left\{ -2b Y^h+(a+c)Y^{t_G}-a (R_u Y)^{t_G}\right\}
\end{aligned}
\right.
\end{equation}
for all $X,Y \in M_x$, with $Y$ orthogonal to $u$.

The paracontact metric structure $(\widetilde\varphi,\widetilde\xi,\widetilde\eta,\widetilde{G})$ is $K$-contact when $\widetilde{h}=0$, that is, by \eqref{h-explicit}, if and only if
\begin{equation*}
\left\{
\begin{aligned}
& -(a+c)X^h+a(R_uX)^h-2b (R_u X)^{t_G} =0,\\
&-2b Y^h+(a+c)Y^{t_G}-a (R_u Y)^{t_G } =0,
\end{aligned}
\right.
\end{equation*}
for all $X,Y \in M_x$ orthogonal to $u$. 
Since the horizontal and vertical lifts are uniquely determined, the above system holds only when both horizontal and tangential parts are zero, whence $b=0$ and $-(a+c)X+a(R_uX)=0$, for all $X\in M_x$ orthogonal to $u$.

Since $b=0$, equation \eqref{para-cond} yields $\alpha=a(a+c)<0$. Hence, $R_u X=\frac{a+c}{a} X$ for all $X$ orthogonal to $u$, where $\frac{a+c}{a}<0$.
Therefore, the Jacobi operator  $R_u$ has just one constant eigenvalue $\frac{a+c}{a}<0$ (besides $0$) and it is well known that this is equivalent to the fact that $(M,\langle ,\rangle )$ has negative constant sectional curvature $\bar{c}=\frac{a+c}{a}$.

\smallskip
\lq\lq(iii)$\Rightarrow$(i)\rq\rq:  We will now assume that $(M,\langle ,\rangle )$ has constant sectional curvature $\bar c$ and consider a $g$-natural paracontact metric structure $(T_1M, \widetilde\varphi,\widetilde\xi,\widetilde\eta,\widetilde{G})$, with $\widetilde G$ satisfying $b=0\neq a$  and $c=(\bar c -1)a$. Since both these conditions are invariant under $\mathcal D$-homothetic deformations, we can again use Remark~\ref{phiprime} to restrict to the case when $\phi >0$ without loss of generality.

We first notice that by \eqref{h-explicit} we now have at once $\widetilde h=0$, that is, $(T_1M, \widetilde\varphi,\widetilde\xi,\widetilde\eta,\widetilde{G})$ is $K$-paracontact. Next, we check that Equation~\eqref{parasasakian} holds for all vector fields $Z,W$ tangent $T_1M$. 

If $Z=W=\w{\xi}$, it follows from \eqref{h-properties} that
\begin{align*}
(\w{\nabla}_{\xi} \w{\varphi})\w{\xi}=0=-\w{G}(\w{\xi},\w{\xi})\w{\xi}+\w{\eta}(\w{\xi}) \w{\xi}.
\end{align*}

If $Z=X^h$ or $Z=X^{t_G}$, where $X$ is orthogonal to $u$, and $W=\w{\xi}$, then applying again \eqref{h-properties} (and taking into account $\w{h}=0$), we find
\begin{align*}
(\w{\nabla}_{Z} \w{\varphi})\w{\xi}=\w{\varphi}^2 Z=Z=\w{G}(Z,\w{\xi})\w{\xi}+\w{\eta}(\w{\xi})Z.
\end{align*}
Finally, a direct calculation shows that \eqref{parasasakian} also holds when taking $Z,W= X^h, Y^{t_G}$, with $X,Y$ orthogonal to $u$, which ends the proof.
\end{proof}

\medskip

The second equation in \eqref{h-properties} (together with the paracontact metric condition $\Phi=d\eta$) implies that the tensor $h$ of any paracontact metric structure $(\varphi,\xi,\eta,g)$ is self-adjoint with respect to $g$, just like in the contact metric case. However, since now $g$ is pseudo-Riemannian, the fact that $h$ is self-adjoint does not imply that $h$ admits an orthonormal basis of eigenvectors. Indeed, it may happen that $h^2=0$ although $h \neq 0$.

We explicitly remark that if $(\varphi _t,\xi _t,\eta _t,g_t)$ is the $\mathcal D_t$-homothetic deformation of $(\varphi,\xi,\eta,g)$, then we have
\[
h_t = \frac 12 \mathcal L _{\xi _t} \varphi _t =\frac{1}{2t} \mathcal L _{\xi } \varphi  =\frac 1t h,
\]
from which it follows at once that \emph{the conditions $h^2 =0\neq h$ are  invariant under $\mathcal D$-homothetic deformations}.
We will now classify $g$-natural paracontact metric structures
with $\widetilde{h}^2=0 \neq \widetilde{h}$, which do not have any contact metric counterpart, since in the contact metric case $h^2=0$ is equivalent to $h=0$, due to the diagonalisability of $h$. Since the existence of these structures is related to the base manifold being Ossermann, we briefly report some information on Osserman manifolds and rank-one symmetric spaces.

\emph{Rank-one symmetric spaces} are  $\mathbb R \mathbb P ^n$, $\mathbb  S^n$, $\mathbb C \mathbb P^n$, $\mathbb H \mathbb  P^n$, $Cay \mathbb P^2$ and their non-compact duals. In the cases of constant sectional curvature there exists just one eigenvalue for the Jacobi operator, and two  eigenvalues $\lambda_1$ and $\lambda_2=4\lambda_1$ in the remaining cases (see for example \cite{BoPV}). In all cases, the eigenvalues of $R_u$ have the same sign: positive in the compact cases, and negative for the non-compact ones.

A Riemannian manifold $(M,\langle,\rangle)$ is called \emph{globally Osserman} if the
eigenvalues of $R_u$ are independent of both the unit tangent
vector $u \in M_x$ and the point $x \in M$. The well-known Osserman conjecture states
that any globally Osserman manifold is locally isometric to a \emph{two-point homogeneous
space}, that is, either a flat space or a rank-one symmetric space.

The Osserman conjecture has been proved in any dimension $n \neq 16$ (\cite{Ch},\cite{N1},\cite{N2}). Moreover, also in dimension $16$, if
$R_u$ admits at most two distinct  constant eigenvalues (besides 0), then the Riemannian manifold $(M,\langle,\rangle)$ is locally isometric to a two-point homogeneous space \cite{N3}.

We are now ready to prove the following result.

\begin{Th}\label{h2=0}
A $g$-natural paracontact metric structure $(\widetilde\varphi,\widetilde\xi,\widetilde\eta,\widetilde{G})$, constructed from a $g$-natural metric $\widetilde G$ with $a \neq 0$, satisfies $\widetilde{h}^2 =0 \neq \widetilde{h}$ if and only if $(M,\langle,\rangle)$ is locally isometric to a non-compact rank-one symmetric space of non-constant sectional curvature and either $\alpha =-b^2/9$, or $\alpha=-9 b^2$.
\end{Th}

\begin{proof}
Since the condition $h^2 =0 \neq h$ is invariant under $\mathcal D$-homothetic deformations, by Remark~\ref{phiprime} it is enough to consider the case when $\widetilde G$ satisfies $\phi >0$. So, using the description of $\widetilde h$ given in Equation \eqref{h-explicit}, we have
\begin{equation*}
\left\{
\begin{aligned}
\widetilde{h}^2(X^h)     &=\frac{1}{16\alpha^2} \{ a^2 R_u (R_u X)-2(a(a+c)-2b^2)R_u X+(a+c)^2 X \}^h,\\
\widetilde{h}^2(Y^{t_G}) &=\frac{1}{16\alpha^2} \{ a^2 R_u (R_u Y)-2(a(a+c)-2b^2)R_u Y+(a+c)^2 Y \}^{t_G},
\end{aligned}
\right.
\end{equation*}
for all $X,Y \in M_x$ orthogonal to $u$. Therefore, $\widetilde{h}^2=0$ if and only if
\begin{equation*}
a^2 R_u (R_uX)-2(a(a+c)-2b^2)R_u X+(a+c)^2 X =0,
\end{equation*}
for all $X \in M_x$ orthogonal to $u$.

We now consider an eigenvector $X \neq 0$ of the Jacobi operator $R_u $, associated to the eigenvalue $\lambda$. Then, the above equation yields
\[
a^2 \lambda^2-2(a(a+c)-2b^2)\lambda+(a+c)^2=0
\]
and so, $R_u$ has at most two constant eigenvalues, explicitly given by 
\begin{equation}\label{eigenvalues}
\lambda=\frac{\alpha-b^2\pm\sqrt{-4\alpha b^2}}{a^2},
\end{equation}
where we took into account the definition of  $\alpha$.


Since $\alpha<0$, Equation~\eqref{eigenvalues} admits one or two real solutions, depending on whether $b=0$ or $b\neq0$.

If $b=0$, then  the only eigenvalue is $\lambda=\frac{a+c}{a}$, so $(M,\langle,\rangle)$ has constant sectional curvature $\frac{a+c}{a}<0$. Thus, by Theorem~\ref{th-parasasakian} we have that $\widetilde{h}=0$ and $(T_1M, \widetilde\varphi,\widetilde\xi,\widetilde\eta,\widetilde{G})$ is paraSasakian.

On the other hand, if $b\neq0$, then there exist two distinct solutions, which can be written as
\[
\lambda=\frac{\alpha-b^2\pm\sqrt{-4\alpha b^2}}{a^2}=\frac{\alpha-b^2\pm2|b|\sqrt{-\alpha}}{a^2}=-\frac{(\sqrt{-\alpha}\pm |b|)^2}{a^2}<0.
\]
Therefore, $(M,\langle,\rangle)$ does not have constant sectional curvature. Since the Jacobi operator $R_u$ has two distinct eigenvalues, the same ones for any unit vector $u$ and at each point, the base manifold is locally isometric to a rank-one symmetric space (non-compact, as $\lambda_i<0$). Requiring that one of these eigenvalues is four times the other, we get
\[
4 \cdot \frac{\alpha-b^2+\sqrt{-4\alpha b^2}}{a^2}=\frac{\alpha-b^2-\sqrt{-4\alpha b^2}}{a^2},
\]
that is, $9\alpha^2+82b^2\alpha+9b^4=0$, whose solutions are either $\alpha=-\frac{1}{9} b^2<0$, or $\alpha=-9b^2<0$ and this ends the proof. Notice that, since $\alpha=a(a+c)-b^2$, we respectively get $a+c+d=-4\alpha=\frac{4}{9}b^2>0$ or $36b^2>0$, compatibly with condition~\eqref{para-cond}.
\end{proof}

We end this section with the following consequence of Theorems~\ref{th-parasasakian} and \ref{h2=0}.

\begin{Cor}
(1) For every manifold $(M,\langle,\rangle)$ of constant sectional curvature $\bar c <0$, there exists a one-parameter family of $g$-natural paraSasakian structures $(\widetilde\varphi,\widetilde\xi,\widetilde\eta,\widetilde{G})$ on $T_1 M$: the ones described by \eqref{xi}-\eqref{para-cond}, constructed from any $g$-natural metric given by \eqref{g-tilde2}, with
\[
a \neq 0, \quad  b=0, \quad c=(\bar c -1)a, \quad d=-\bar c a(4a+1).
\]

\smallskip
(2) For every non-compact rank-one symmetric space $(M,\langle,\rangle)$ of non-constant sectional curvature, there exist two two-parameter families of $g$-natural paracontact metric structures $(\widetilde\varphi,\widetilde\xi,\widetilde\eta,\widetilde{G})$ on $T_1 M$, satisfying $\widetilde h ^2 =0 \neq \widetilde h$: the ones described by \eqref{xi}-\eqref{para-cond}, constructed from any $g$-natural metric given by \eqref{g-tilde2}, with
\[
a \neq 0, \quad  b\neq 0, \quad d=-4 \alpha -(a+c), \quad \alpha =a(a+c)-b^2= -\frac{1}{9} b^2 \; {\rm or} \; -9b^2.
\]
\end{Cor}

\section{Paracontact $g$-natural $(\kappa,\mu)$-spaces}\label{sec-kappamu}

We will now consider the remarkable class of \emph{paracontact metric $(\kappa,\mu)$-spaces}, which are paracontact metric manifolds $(M,\varphi,\xi,\eta,g)$ satisfying the condition
\begin{equation} \label{kappamu}
{R}(Z,W){\xi}=\kappa ({\eta}(W)Z-{\eta}(Z)W)+\mu({\eta}(W){h}Z-{\eta}(Z){h}W),
\end{equation}
for all vector fields $Z,W$ on $M$, where  $\kappa$ and $\mu$ are real constants. These manifolds are a natural generalisation of both the paracontact metric manifolds satisfying ${R}(X,Y){\xi}=0$  and of the paraSasakian ones.

Their analogue in contact metric geometry (namely, contact metric $(\kappa,\mu)$-spaces)  where first introduced and studied in \cite{BKP}. Much more recently, paracontact metric $(\kappa,\mu)$-spaces have been studied in \cite{CKM}. In particular, we report the following result.

\begin{Lem}[\cite{CKM}]
If $(M,\varphi,\xi,\eta,g)$ is a paracontact metric $(\kappa,\mu)$-space, then
\begin{equation}\label{h-kappamu}
 {h}^2=(1 + \kappa){\varphi}^2.
\end{equation}
Moreover, when $\kappa\neq -1$ the following identity holds:
\begin{equation}\label{nabla-kappamu}
({\nabla}_{Z} {\varphi}){W}=-g(Z-hZ,W)\xi+\eta(W)(Z-hZ),
\end{equation}
for all vector fields $Z,W$ on $M$.
\end{Lem}

Equation \eqref{h-kappamu} means that $\kappa=-1$ if and only if ${h}^2=0$ (which, as already remarked in  the previous section, does not imply ${h}=0$).

As already proved in \cite{CKM}, if $(\varphi _t,\xi _t,\eta _t,g_t)$ is the $D_t$-homothetic deformation of a paracontact metric $(\varphi,\xi,\eta,g)$, then $(\varphi _t,\xi _t,\eta _t,g_t)$ is a $(\kappa,\mu)$-space if and only if so is if $(\varphi ,\xi ,\eta ,g)$. In particular, one has \cite{CKM}
\[
\kappa _t =\frac{k+1-t^2}{t^2}, \qquad \mu _t= \frac{\mu+2t-2}{t},
\]
so that $\kappa _t=-1$ if and only if $k=-1$. We are now ready  to prove the following result.


\begin{Th}\label{th-kappamu}
Consider any $g$-natural paracontact metric structure $(\widetilde\varphi,\widetilde\xi,\widetilde\eta,\widetilde{G})$ on $T_1 M$, constructed from a $g$-natural metric $\widetilde G$ with $a \neq 0$. Then:
\begin{itemize}
\item[(a)]
If $(T_1 M,\widetilde\varphi,\widetilde\xi,\widetilde\eta,\widetilde{G})$ is a $(\kappa,\mu)$-space with $\kappa\neq -1$, then $(M,\langle,\rangle)$ is of constant sectional curvature $\bar{c}$.

\item[(b)]
If $(M,\langle,\rangle)$ is of constant sectional curvature $\bar{c}$, then $(T_1 M,\widetilde\varphi,\widetilde\xi,\widetilde\eta,\widetilde{G})$ is a paracontact $(\kappa,\mu)$-space. Moreover, if it is not paraSasakian, then
\begin{equation}\label{values}
\kappa=\frac{1}{16\alpha^2} ( a^2\bar{c}^2 -2(\alpha-b^2)\bar{c}-d(2(a+c)+d)) \neq -1, \quad \mu=\frac{1}{2\alpha} (a \bar{c}-d).
\end{equation}
\end{itemize}
\end{Th}
\begin{proof}
As conditions $\kappa \neq -1$ and $a \neq 0$ are invariant under $\mathcal D$-homothetic deformations, without loss of generality we can assume $\phi >0$ (Remark~\ref{phiprime}). Let us suppose that $(T_1M, \widetilde\varphi,\widetilde\xi,\widetilde\eta,\widetilde{G})$ is a paracontact $(\kappa,\mu)$-space with $\kappa\neq -1$.  Then equation \eqref{nabla-kappamu} holds for all vector fields $Z,W$ on $T_1M$.  Taking $Z=Y_1^{t_G}$ and $W=Y_2^{t_G}$, with $Y_1, Y_2$ orthogonal to $u$, the formulas from
Theorem~\ref{th-paracontact} and the formulas for $\w{h}$ in \eqref{h-explicit} give us that
\begin{equation*}
\begin{aligned}
&-\w{G}(Y_1^{t_G}-\w{h}Y_1^{t_G},Y_2^{t_G})\w{\xi} +\w{\eta}(Y_2^{t_G})({Y_1^{t_G}}-\w{h}{Y_1^{t_G}})=-\w{G}({Y_1^{t_G}}-\w{h}{Y_1^{t_G}},Y_2^{t_G})\w{\xi}\\
&=-\w{G}({Y_1^{t_G}},Y_2^{t_G})\w{\xi}+\frac{1}{4\alpha}\w{G}(-2bY_1^h+(a+c)Y_1^{t_G}-a (R_u Y_1)^{t_G},Y_2)\w{\xi}.
\end{aligned}
\end{equation*}
Using \eqref{g-tilde2}  and the definition of $\alpha$, we then obtain
\begin{equation}\label{eq-first}
\begin{array}{l}
-\w{G}(Y_1^{t_G}-\w{h}Y_1^{t_G},Y_2^{t_G})\w{\xi}+\w{\eta}(Y_2^{t_G})({Y_1^{t_G}}-\w{h}{Y_1^{t_G}})
=\left(  \left(-a+\frac{\alpha-b^2}{4\alpha} \right) \langle Y_1,Y_2 \rangle-\frac{a^2}{4\alpha} \langle R_u Y_1,Y_2 \rangle \right) \rho u^h.
\end{array}
\end{equation}
On the other hand, it follows from Proposition~\ref{lcc-T1M} that
\begin{equation}\label{eq-second}
\begin{aligned}
(\w{\nabla}_{Y_1^{t_G}} \w{\varphi}){Y_2^{t_G}}&=\w{\nabla}_{Y_1^{t_G}} (\w{\varphi}{Y_2^{t_G}})-\w{\varphi}(\w{\nabla}_{Y_1^{t_G}} {Y_2^{t_G}})=\w{\nabla}_{Y_1^{t_G}} (\w{\varphi}{Y_2^{t_G}})\\
&=\frac{1}{2\rho\alpha} \left(-a\w{\nabla}_{Y_1^{t_G}} Y_2^h +\frac{\phi}{a+c+d} \w{\nabla}_{Y_1^{t_G}} (\langle Y_2,u \rangle) u^h \right)\\
&=\frac{a^3}{4\rho\alpha^2} (R(Y_1,u)Y_2)^h +\left(  \frac{2\alpha+ad}{4\alpha} \langle Y_1,Y_2 \rangle-\frac{a^2(ad+b^2)}{4\alpha^2} \langle R(Y_1,u)Y_2,u \rangle \right) \rho u^h\\
&-\frac{a^2b}{4\rho\alpha^2} (R(Y_1,u)Y_2- \langle R(Y_1,u) Y_2,u \rangle u)^{t_G},
\end{aligned}
\end{equation}
where we have used the fact that
\begin{align*}
\w{\nabla}_{Y_1^{t_G}} (\langle Y_2,u \rangle) &=\frac{1}{\rho(a+c+d)} \w{\nabla}_{Y_1^{t_G}} ( \w{G} (Y_2^h,\w{\xi} ))=\langle Y_1,Y_2 \rangle.
\end{align*}

Comparing equations \eqref{eq-first} and \eqref{eq-second}, we get
\[
a^3 \{\langle R(Y_1,u)Y_2,u \rangle u-R(Y_1,u)Y_2 \}^h-a^2b \{ \langle R(Y_1,u) Y_2,u \rangle u-R(Y_1,u)Y_2 \}^{t_G}=0.
\]
Since $a\neq 0$, the above equation yields $R(X,u)Y=\langle R(X,u)Y,u \rangle u$, for all vector fields $X,Y$ on $M$ orthogonal to $u$. Therefore, $(M,\langle,\rangle)$ is of constant sectional curvature \cite{blair-book}.

Conversely, we  now suppose that $(M,\langle,\rangle)$ is of constant sectional curvature $\bar{c}$ and we want to check that formula \eqref{kappamu} is satisfied for some values of $\kappa$ and $\mu$. Indeed, we have from Proposition~\ref{riem-curv-T1M} that
\[
\w{R}(X_1^h,X_2^h)\w{\xi}=\w{R}(X_1^h,Y_1^{t_G})\w{\xi}=\w{R}(Y_1^{t_G},Y_2^{t_G})\w{\xi}=0,
\]
for all $X_i,Y_i$ orthogonal to $u$. 
By the symmetries of the curvature tensor $\w{R}$, it is then enough to check that \eqref{kappamu} holds for $Z=X^h$ (or $Y^{t_G}$) and $W=\w{\xi}$, with $X,Y$ orthogonal to $u$. Applying Proposition~\ref{riem-curv-T1M}, we have that
\begin{align}
\w{R}(X^h,\w{\xi})\w{\xi} &=\frac{\rho^2}{4\alpha} (-3a^2 \bar{c}^2+(4\alpha+2ad)\bar{c}+d^2 )X^h
+\frac{\rho^2}{\alpha}X^{t_G}, \label{r-x-xi}\\
\w{R}(Y^{t_G},\w{\xi})\w{\xi} &=\frac{\rho^2}{\alpha}(ab\bar{c}-bd)Y^h
+\frac{\rho^2}{4\alpha}(a^2\bar{c}^2+2(ad+2b^2)\bar{c}+d(4(a+c)+d))Y^{t_G}. \label{r-y-xi}
\end{align}
On other hand, we have from \eqref{kappamu} that
\begin{align}
\kappa (\w{\eta}(\w{\xi})X^h-\w{\eta}(X^h)\w{\xi}) &+\mu(\w{\eta}(\w{\xi})\w{h}X^h-\w{\eta}(X^h)\w{h}\w{\xi}) \label{r-x-xi2}\\
&=\left( \kappa+\frac{a\bar{c}-(a+c)}{4\alpha} \right)X^h-\frac{b \mu \bar{c}}{2\alpha} X^{t_G} \nonumber\\
\kappa (\w{\eta}(\w{\xi})Y^{t_G}-\w{\eta}(Y^{t_G})\w{\xi}) &+\mu(\w{\eta}(\w{\xi})\w{h}Y^{t_G}-\w{\eta}(Y^{t_G})\w{h}\w{\xi}) \label{r-y-xi2}\\
&=-\frac{b \mu}{2\alpha} Y^h+\left( \kappa-\frac{a\bar{c}-(a+c)}{4\alpha} \mu \right)Y^{t_G} \nonumber
\end{align}

Therefore, formula \eqref{kappamu} is satisfied if and only if \eqref{r-x-xi} and \eqref{r-x-xi2} coincide, as well as \eqref{r-y-xi} and \eqref{r-y-xi2}. Since the horizontal and tangential parts are uniquely determined, this is equivalent to the following system of equations:
\begin{equation}\label{system-kappamu}
\left\{
\begin{aligned}
&3a^2\bar{c}^2-(4\alpha+2ad)\bar{c}-d^2=16\alpha^2 \kappa+4\alpha(a\bar{c}-(a+c))\mu,\\
&ab\bar{c}^2-bd\bar{c}=2\alpha b\mu \bar{c},\\
&ab\bar{c}-bd=2\alpha b \mu,\\
&-a^2\bar{c}^2+2(ad+2b^2)\bar{c}-d(4(a+c)+d)=16\alpha^2\kappa-4\alpha(a\bar{c}-(a+c))\mu,
\end{aligned}
\right.
\end{equation}
where we took into account \eqref{para-cond}. Summing the first and last equations of \eqref{system-kappamu}, we obtain
\[
a^2 \bar{c}^2 -2(\alpha-b^2)\bar{c}-d^2-2(a+c)d=16\alpha^2 \kappa,
\]
which determines $\kappa$ as in \eqref{values}.

The second equation  of \eqref{system-kappamu} follows form the third one and we have that either $b=0$ or $\mu=\frac{1}{2\alpha}(a\bar{c}-d)$, as in \eqref{values}. Thus, we are left with the case $b=0$.

If $b=0$ and $\bar{c} = \frac{a+c}{a}$, then $T_1 M$ is paraSasakian (see Theorem~\ref{th-parasasakian}) and in particular a $(\kappa,\mu)$-space. If $b=0$ and $\bar{c} \neq \frac{a+c}{a}$ (that is, $T_1M$ is not paraSasakian) then substituting the value obtained for $\kappa$ in the first equation of \eqref{system-kappamu}, we find
\[
a^2\bar{c}^2-a(a+c+d)\bar{c}+(a+c)d=2\alpha(a\bar{c}-(a+c))\mu.
\]
Therefore,
\[
\mu=\frac{1}{2\alpha} \frac{a^2\bar{c}^2-a(a+c+d)\bar{c}+(a+c)d}{a\bar{c}-a(a+c)}=\frac{1}{2\alpha} \frac{(a\bar{c}-d)(a\bar{c}-(a+c))}{a\bar{c}-a(a+c)}=\frac{1}{2\alpha} (a\bar{c}-d),
\]
obtaining again \eqref{values}.

We will now show that $\kappa\neq -1$ when the $g$-natural paracontact metric structure is not paraSasakian. We will prove this by contradiction. Let us suppose that $\kappa=-1$ but the $g$-natural paracontact metric structure is not paraSasakian. Then, \eqref{values} implies that
\[
1+\kappa=\frac{1}{16\alpha^2} (a^2\bar{c}^2 -2(\alpha-b^2)\bar{c}+(a+c)^2),
\]
where we have used the definition of $\alpha$, \eqref{para-cond} and the fact that $d=-(a+c)(4a+1)+4b^2$. Hence, $\kappa=-1$ if and only if $a^2\bar{c}^2 -2(\alpha-b^2)\bar{c}+(a+c)^2=0$, whose solutions are  $\bar{c}=-\left( \frac{\sqrt{-\alpha} \pm |b|}{a}\right)^2$.

If $b=0$, then $\bar{c}=-\left( \frac{\sqrt{-\alpha}}{a}\right)^2=\frac{\alpha}{a^2}=\frac{a+c}{a}$. So, the $g$-natural paracontact metric structure  is paraSasakian (Theorem~\ref{th-parasasakian}), which  contradicts our assumption.

If $b\neq0$, then $\w{h}\neq0$ (Theorem~\ref{th-parasasakian}) and $\w{h}^2=0$ (formula  \eqref{h-kappamu}). But then, as proved in Theorem~\ref{h2=0}, $M$ cannot be of constant sectional curvature. So, this case cannot occur, either.
\end{proof}

\begin{Rk}
We showed in the proof of the above Theorem that if the base manifold $(M,\langle,\rangle)$ is of constant sectional curvature $\bar{c}$ and $T_1M$  is a non-paraSasakian $(\kappa,\mu)$-space, then $\kappa\neq-1$. However, this does not exclude the existence of non-paraSasakian $g$-natural paracontact metric manifolds $(T_1M, \widetilde\varphi,\widetilde\xi,\widetilde\eta,\widetilde{G})$ that are $(\kappa,\mu)$-spaces with $\kappa= -1$, it only ensures that their base manifold cannot have constant sectional curvature.
\end{Rk}

\medskip
%
%
We now characterise  paracontact $(\kappa,\mu)$-spaces  of constant $\varphi$-sectional curvature.
Koufogiorgos \cite{kouf97} proved that a $(2n+1)$-dimensional $(n>1)$, non-Sasakian $(\kappa, \mu)$-contact Riemannian manifold $(M, \varphi, \xi, \eta, g)$ has constant $\varphi$-sectional curvature if and
only if $\mu=1+\kappa$.

As a consequence (see again \cite{kouf97}), if $M$ is an $n$-dimensional Riemannian manifold, $n>2$, of constant
sectional curvature $c$, then the tangent sphere bundle $T_{1}M$ has constant $\varphi$-sectional curvature
$(c^{2})$ if and only if $c=2\pm\sqrt{5}$. On the other hand, if $n=2$ then the tangent sphere bundle $T_{1}M$ always has constant $\varphi$-sectional curvature $c^{2}$ for any $c\neq 1$.

A three-dimensional paracontact metric $(\kappa,\mu)$-space has always constant $\varphi$-sectional curvature.
The paracontact analogue of Koufogiorgos' result is given by the following.

\begin{Th}\label{space-forms}
Let $(M, \varphi,\xi,\eta,g)$ be a paracontact metric $(\kappa,\mu)$-space of dimension $2n+1\geq5$ with $\kappa\neq-1$. Then it has constant $\varphi$-sectional curvature if and only if $\mu=1-\kappa$.
\end{Th}
\begin{proof}
The curvature tensor of a paracontact $(\kappa,\mu)$-space with $\kappa \neq -1$ was completely described in \cite{CKM}. In particular, for any non-lightlike vector field $X$ orthogonal to $\xi$, we can compute the corresponding $\varphi$-sectional curvature as
\begin{equation}\label{phi-sectional}
K(X,\varphi X)=2\mu-1-\frac{\kappa-1+\mu}{\kappa+1} \cdot \frac{g(hX,X)^2-g(\varphi h X,X)^2}{g(X,X)^2}.
\end{equation}

If $\mu=1-\kappa$, this means that $K(X,\varphi X)=2\mu-1$, which is constant.

Conversely, if $K(X,\varphi X)$ is constant, then either $\kappa-1+\mu=0$ (and so, $\mu=1-\kappa$), or $\frac{g(hX,X)^2-g(\varphi h X,X)^2}{g(X,X)^2}$ does not depend on the vector field $X$. We will see that the latter case is not possible.

Indeed, we know from \cite{CKM} that if $\kappa>-1$ there exists a $\varphi$-basis $\{ \xi, e_1, \ldots, e_n, \varphi e_1,\ldots,\varphi e_n\}$ such that $h e_i=\lambda e_i$ and $h \varphi e_i=-\lambda \varphi e_i$, where $\lambda=\sqrt{1+\kappa}$ and $g(e_i,e_i)=-g(\varphi e_i, \varphi e_i)=\pm 1$ (how many of each sign will depend on the index of the eigendistributions of $h$). If $\kappa<-1$, we can take a $\varphi$-basis of eigenvectors of $\varphi h$ with eigenvalues $\pm \lambda=\pm \sqrt{-(1+\kappa)}$.

If we take $X=e_1$, then $X$ is orthogonal to $\xi$ and $g(X,X)=g(e_1,e_1)=\pm 1\neq0$, so
\begin{equation*}
\frac{g(hX,X)^2-g(\varphi h X,X)^2}{g(X,X)^2}=
\left\{
\begin{aligned}
\lambda^2&=1+\kappa, \text{ if } \kappa>-1,\\
-\lambda^2&=1+\kappa, \text{ if } \kappa<-1.
\end{aligned}
\right.
\end{equation*}

On the other hand, if we take $X=e_1+2\varphi e_2$, then it is also orthogonal to $\xi$ and
\[
g(X,X)=g(e_1,e_1)+4g(\varphi e_2,\varphi e_2)=g(e_1,e_1)-4g(e_2, e_2)\neq0,
\]
so we can compute again:
\begin{equation*}
\frac{g(hX,X)^2-g(\varphi h X,X)^2}{g(X,X)^2}=
\left\{
\begin{aligned}
&\lambda^2 \left( \frac{g(e_1,e_1)+4g(e_2, e_2)}{g(e_1,e_1)-4g(e_2, e_2)} \right)^2 \neq\lambda^2=1+\kappa, \text{ if } 1+\kappa>-1,\\
&-\lambda^2 \left( \frac{g(e_1,e_1)+4g(e_2, e_2)}{g(e_1,e_1)-4g(e_2, e_2)} \right)^2 \neq -\lambda^2=1+\kappa, \text{ if } 1+\kappa<-1.
\end{aligned}
\right.
\end{equation*}
Since the above values do not coincide, the $\varphi$-sectional curvature cannot be constant. So, this case cannot occur.
\end{proof}

\noindent
As a consequence of Theorems~\ref{th-kappamu} and \ref{space-forms}, we have the following result.

\begin{Cor}\label{kmuphi}
If $(M, \langle,\rangle)$ is a $n$-dimensional ($n\geq 2$) Riemannian manifold with constant sectional curvature $\bar{c}$, then the $g$-natural paracontact metric manifold $(T_1 M,\w\phi,\w\xi,\w\eta,\w{G})$ defined as in Theorem~\emph{\ref{th-paracontact}} is a $(\kappa,\mu)$-space of constant $\w{\varphi}$-sectional curvature if and only if $\bar c$ and the parameters $a \neq 0,b,c,d$ determining $\widetilde G$ satisfy
\begin{equation}\label{costphi}
a^2 \bar{c}^2+2((4a-1)\alpha+b^2)\bar{c}-(a+c)(a+c+2d)=0.
\end{equation}
In particular, if $(M,\langle,\rangle)$ is flat, then
\begin{itemize}
\item either $b=\pm \sqrt{(a+c)(a+\frac18)}$ and $d=-\frac{a+c}2<0$,
\item  or $a+c=0$, $b\neq0$ and $d=4b^2>0$.
\end{itemize}
\end{Cor}

\begin{Rk}
It is easily seen that Equation~\eqref{costphi} is incompatible with conditions $b=0, \bar c=\frac{a+c}{a}$ characterizing $g$-natural structures (Theorem~\ref{th-parasasakian}). Hence, the examples described in Corollary~\ref{kmuphi} are $(\kappa,\mu)$-spaces, not paraSasakian, of constant $\varphi$-sectional curvature. ParaSasakian space forms (that is, paraSasakian manifolds with constant $\varphi$-sectional curvature) were classified in \cite{zamkovoy-arxiv}.
\end{Rk}

\section{Paracontact $g$-natural $(\kappa,\mu)$-spaces from the contact ones}\label{sec-deform}

We will now investigate the relationship between $g$-natural contact metrics on $T_1M$, as introduced and first studied in \cite{AC1}, and $g$-natural paracontact metric structures, considering the deformations of a contact $(\kappa,\mu)$-space into a paracontact one introduced in \cite{C}.

We recall that a $g$-natural contact metric structure on $T_1M$, which we will denote here by $(\varphi ',{\xi}',{\eta}',{G}')$, is a contact metric structure with ${G}'$ a \emph{Riemannian} metric induced by a $g$-natural metric $G$ on $TM$ (see \cite{AC1}). The Riemannian metric $G'$ is explicitly described by \eqref{g-tilde2}, for some real parameters $a' ,b',c',d'$, satisfying conditions \eqref{cond-Riem-T1M}, that is $a'>0,  a'+c'+d'>0$ and $\alpha '=a '(a'+c' )-(b') ^2>0$. Notice that $\alpha'$ has the opposite sign with respect to its analogue for $g$-natural paracontact metric structures. The Reeb vector field is given by $\xi '=\rho' u^h$, and the tensor $\varphi '$ and the $1$-form  $\eta '$, defined as in the formulas (3.6)-(3.8) of \cite{AC1}, formally coincide with \eqref{xi}-\eqref{phi} of the $g$-natural paracontact metric structures. The compatibility condition corresponding to \eqref{para-cond} is given by
\[
a'+c'+d'=\frac{1}{(\rho') ^2}=4\alpha'.
\]

Following \cite{C}, for any contact metric $(\kappa,\mu)$-space $(\varphi ', \xi ',\eta ', g')$ which is not Sasakian (that is, satisfies $\kappa<1$), we can define two canonical paracontact metric structures, by taking
\begin{align}
{\varphi}_1 =& \frac{1}{\sqrt{1-\kappa}}\varphi ' h', &\quad  &{g}_1=\frac{1}{\sqrt{1-\kappa}} d \eta '(\cdot, \varphi 'h' \cdot)+ \eta '\otimes \eta '
=-\frac{1}{\sqrt{1-\kappa}} g' (\cdot, h' \cdot)+\eta '\otimes \eta ',  \label{deform-1}\\
{\varphi}_2 =& \frac{1}{\sqrt{1-\kappa}}h', &\quad &{g}_2=\frac{1}{\sqrt{1-\kappa}} d\eta'(\cdot, h' \cdot)+\eta '\otimes \eta '
=\frac{1}{\sqrt{1-\kappa}} g (\cdot,\varphi ' h' \cdot)+\eta '\otimes \eta '.  \label{deform-2}
\end{align}
Moreover, the deformed structures $(\varphi_1,\xi', \eta',g_1)$  and $(\varphi_2,\xi', \eta',g_2)$ are paracontact metric $(\kappa_i,\mu_i)$-spaces, with
\begin{align*}
\kappa_1&=\left(1-\frac{\mu}2 \right)^2-1, & \mu_1&=2(1-\sqrt{1-\kappa}),\\
\kappa_2&=\kappa-2 +\left(1-\frac{\mu}2 \right)^2, & \mu_2&=2.
\end{align*}

As proved in \cite{AC2}, if  $(T_1M,{\varphi}', {\xi}',{\eta}',{G}')$ is a $g$-natural non-Sasakian $(\kappa,\mu)$-space, then $(M,\langle,\rangle)$ is of constant sectional curvature $\bar{c}$,  and
\[
\kappa=\frac{1}{16 (\alpha')^2} \left( -(a')^2\bar{c}^2 +2( \alpha '-(b')^2)\bar{c}+d'(2(a'+c')+d')\right), \quad \mu=\frac{1}{2\alpha'} (d'-a' \bar{c}).
\]
Therefore,
\begin{equation}\label{1-kappa}
1-\kappa=\frac{1}{16 (\alpha ')^2} \left((a')^2\bar{c}^2 -2(\alpha'-(b')^2)\bar{c}+(a'+c')^2\right).
\end{equation}

If we now deform these structures as in \eqref{deform-1} and \eqref{deform-2}, we obtain two new paracontact metric $(\kappa,\mu)$-structures on $T_1 M$, which we will denote by $(\w{\varphi}_1,\xi',\eta',\w{G}_1)$ and $(\w{\varphi}_2,\xi',\eta',\w{G}_2)$. We will now prove that these new structures are indeed $g$-natural paracontact metric structures, by checking that they satisfy the conditions \eqref{xi}-\eqref{para-cond} given in Theorem~\ref{th-paracontact}.

Indeed, let us suppose that $(\w{\varphi}_1,\xi',\eta',\w{G}_1)$ is a $g$-natural paracontact metric structure for some real constants $a_1,b_1,c_1,d_1$. Then, formulas \eqref{g-tilde2} and \eqref{deform-1} respectively give
\begin{align*}
\w{G}_1(u^h,u^h)=a_1+c_1+d_1,\quad
\w{G}_1(u^h,u^h)=\frac{1}{\sqrt{1-\kappa}}  \w{G}(u^h,\w{h} u^h)+\w{\eta}(u^h)\w{\eta}(u^h)=\frac{1}{\rho^2}.
\end{align*}
Hence,
\begin{equation}\label{deform-cond-1}
a_1+c_1+d_1=\frac{1}{\rho^2}=\frac{1}{(\rho ')^2}=a'+c'+ d'.
\end{equation}

Formula \eqref{g-tilde2}, applied on the pairs $(u^h,X^h)$ and $(u^h,Y^{t_G})$, with $X,Y$ orthogonal to $u$, does not give any extra conditions. We will now consider $\w{G}_1(X_1^h,X_2^h)$, with with $X_1$ and $X_2$ orthogonal to $u$. Then, from \eqref{g-tilde2} and \eqref{deform-1} we respectively find
\begin{align*}
&\w{G}_1(X_1^h,X_2^h)=(a_1+c_1) \langle X_1,X_2 \rangle,\\
&\w{G}_1(X_1^h,X_2^h)=-\frac{1}{4\alpha' \sqrt{1-\kappa}} \left(\alpha'-(b')^2\right)\left(\bar{c}-(a'+c')^2\right) \langle X_1,X_2 \rangle.
\end{align*}
So,
\[
a_1+c_1=-\frac{1}{4\alpha' \sqrt{1-\kappa}} \left((\alpha'-(b')^2)\bar{c}-(a'+c')^2\right)
\]
and \eqref{deform-cond-1} yields
\begin{equation}\label{deform-d1}
d_1=a'+c'+d'-(a_1+c_1)=(a'+c'+d')+\frac{1}{4\alpha '\sqrt{1-\kappa}} \left((\alpha'- (b')^2)\bar{c}-( a'+c')^2\right).
\end{equation}

By a similar argument, we obtain
\begin{align*}
\w{G}_1 (X^h,Y^{t_G}) &=b_1 \langle X,Y \rangle=\frac{b'}{4\alpha' \sqrt{1-\kappa}} (a' \bar{c}+ a'+c') \langle X,Y \rangle,\\
\w{G}_1 (Y_1^{t_G},Y_2^{t_G}) &=a_1 \langle Y_1,Y_2 \rangle=\frac{1}{4\alpha' \sqrt{1-\kappa}} ((a')^2\bar{c}-\alpha'+(b')^2) \langle Y_1,Y_2 \rangle,
\end{align*}
for $X,Y,Y_1,Y_2$ orthogonal to $u$. Thus, we get
\begin{equation}\label{deform-a1-b1}
a_1 =\frac{1}{4 \alpha' \sqrt{1-\kappa}}((a')^2\bar{c}-\alpha'+(b')^2), \quad
b_1 =\frac{1}{4 \alpha' \sqrt{1-\kappa}}( a'\bar{c}+a'+c')
\end{equation}
and
\begin{equation}\label{deform-c1}
c_1=(a_1+c_1)-a_1=\frac{1}{4\alpha' \sqrt{1-\kappa}}[(-(a')^2-\alpha'+(b')^2)\bar{c}+\alpha'-(b')^2+(a'+c')^2].
\end{equation}

Substituting the values of $a_1, b_1, c_1$ from \eqref{deform-a1-b1}, \eqref{deform-c1} and \eqref{deform-d1} in the definition of $\alpha_1$ and using \eqref{1-kappa}, we find
\begin{align*}
\alpha_1 &=a_1(a_1+c_1)-b_1^2=-\frac{1}{16\alpha'(1-\kappa)} \left((a')^2 \bar{c}^2-2(\alpha'-(b')^2)\bar{c}+(a'+c')^2\right)=-\alpha'<0.
\end{align*}
In particular, this implies
\[
-4\alpha_1=4\alpha'=\frac{1}{\rho^2}=a'+c'+d'=a_1+c_1+d_1,
\]
so \eqref{para-cond} holds. Finally, it is easy to check that $\w{\varphi}_1$ satisfies \eqref{phi}. Therefore, $(\w{\varphi}_1,\xi',\eta',\w{G}_1)$ is a $g$-natural paracontact metric structure.

\medskip

The proof of $(\w{\varphi}_2,\xi',\eta',\w{G}_2)$ being a $g$-natural paracontact metric structure for constants $a_2$, $b_2$, $c_2$ and $d_2$ is similar to the previous case. Explicitly, we obtain
\begin{align*}
a_2 &=-\frac{\rho}{\sqrt{1-\kappa}} b',                &  b_2 &=-\frac{\rho}{2\sqrt{1-\kappa}} (a'\bar{c}-(a'+c')), \\
 c_2 &=-\frac{\rho}{\sqrt{1-\kappa}}(1+\bar{c})b',     &  d_2 &=\frac{1}{\rho^2}-\frac{\rho}{\sqrt{1-\kappa}}b'\bar{c}.
\end{align*}
Thus, taking into account \eqref{1-kappa}, we find $\alpha_2 =-\alpha'<0$ and
\[
\frac{1}{\rho^2}=a_2+c_2+d_2=-4\alpha_2.
\]
So, all the conditions of Theorem~\ref{th-paracontact} are satisfied. In this way, we proved the following.

\begin{Th}
Let $(M,\langle,\rangle)$ denote a manifold of constant sectional curvature. Then, given a non-Sasakian $g$-natural contact metric structure $({\varphi}', {\xi}',{\eta}',{G}')$ on $T_1 M$ (which is indeed a contact $(\kappa,\mu)$-space),
the canonical paracontact metric structures $(\w{\varphi}_i,\xi',\eta',\w{G}_i)$ described by \eqref{deform-1}, \eqref{deform-2} are $g$-natural paracontact metric $(\kappa,\mu)$-spaces.
\end{Th}

We explicitly remark that for the canonical paracontact metric structures $(\w{\varphi}_i,\xi',\eta',\w{G}_i)$, constructed from a non-Sasakian $g$-natural $(\kappa,\mu)$-space $({\varphi}', {\xi}',{\eta}',{G}')$ on $T_1 M$, all cases are possible with regard to the values of $a_i$ (consequently, of $\phi_i$). In particular, the above formulas yield at once that
\begin{itemize}
\item $a_1=0$ if and only if $\bar c= \frac{a'(a'+c')-2(b')^2}{(a')^2}$;
\vspace{1mm}\item $a_2=0$ if and only if $b'=0$.
\end{itemize}

\section{Homogeneity and harmonicity properties}\label{sec-final}

Similarly to the contact metric case, a paracontact metric manifold $(\bar M, \bar\eta, \bar g)$ is said to be \emph{(locally) homogeneous paracontact} if it admits a transitive (pseudo-)group of (local) isometries leaving invariant its contact form $\bar \eta$ (and hence, the whole paracontact metric structure) \cite{C1}.

As proved in \cite{KS}, the tangent sphere bundle $T_r M$ of any radius $r>0$ of a two-point homogeneous space, equipped with any Riemannian $g$-natural metric, is homogeneous. This result was proved in \cite{KS} for Riemannian $g$-natural metrics, but the argument used does not need the metric to be positive definite. So, the same result is true for pseudo-Riemannian $g$-natural metrics as well. We will now prove the following result.

\begin{Th}\label{homparacont}
Let $(M,\langle,\rangle)$ be (locally isometric to) a two-point homogeneous space. Then, any $g$-natural paracontact metric structure $(\widetilde \varphi , \widetilde \xi, \widetilde \eta, \widetilde G)$ on $T_1M$  is (locally) homogeneous paracontact.
\end{Th}

\begin{proof} Consider any $g$-natural paracontact metric structure $(\widetilde \varphi ,\widetilde \xi,\widetilde  \eta ,\widetilde G)$ on $T_1M$. As $(M,\langle,\rangle)$ is a two-point homogeneous space, $(T_1 M,\widetilde G)$ is homogeneous \cite{KS}. More precisely, following the argument used in \cite{KS}, any (local) isometry $\psi$ of $(M,\langle,\rangle)$ can be lifted to a (local) isometry $\Psi$ of $(T_1M,\widetilde G)$, defined by
\[\Psi (z)=\Psi (x,u)=(\psi (x),\psi _* u),\]
for any unit tangent vector $z=(x,u) \in T_1 M$.

Let $\gamma$ denote the unique geodesic of $(M,\langle,\rangle)$, such that $\gamma (0)=x$ and $\dot{\gamma} (0)=u$. Then, we have $\widetilde \xi_{z}= \rho u^h= \rho \dot{\widetilde{\gamma}} (0)$,
where we put $\widetilde{\gamma} (t):=(\gamma (t), \dot{\gamma} (t))$. Hence,
\begin{equation}\label{fixi}
\Psi _{*z} \widetilde \xi _z = \rho \Psi _{*z} \dot{\widetilde{\gamma}} (0)=\rho \dot{\overbrace{\Psi \circ \widetilde \gamma}} (0).
\end{equation}
Since $\gamma$ and $\psi$  are respectively a geodesic and a local isometry of $(M,\langle,\rangle)$, the curve $\alpha(t):=\psi(\gamma(t))$ is again a geodesic of $(M,\langle,\rangle)$ and, by \eqref{fixi}, the curve
\[
\widetilde{\alpha} (t):=(\Psi \circ \widetilde{\gamma})(t) =(\psi(\gamma (t)),\psi_* \dot{\gamma}(t))
\]
satisfies
\[
\widetilde{\alpha} (0)=\Psi (z), \qquad \dot{\widetilde{\alpha}} (0)=\frac{1}{\rho} \Psi _{*z}\widetilde \xi _{z}.
\]
Hence,
\[
\widetilde \xi _{\psi(z)}=\Psi _{*z} \widetilde \xi _z,
\]
that is, $\widetilde \xi$ is invariant under the isometries of the form $\Psi$, which acts transitively on $(T_1 M,\widetilde G)$. Because of the definitions of tensors $\eta$, $\varphi$ and the paracontact metric condition $\Phi=d\eta$, the invariance of both $\widetilde G$ and $\widetilde \xi$ implies at once that $\Psi$ also leaves invariant $\widetilde \eta$ and $\widetilde \varphi$. Therefore, $(\widetilde \varphi, \widetilde\xi, \widetilde \eta,\widetilde G)$ is a homogeneous paracontact metric structure.
\end{proof}

Taking into account the results of Sections~3 and 4, the above Theorem~\ref{homparacont} yields at once the following result.

\begin{Cor}
$g$-natural paraSasakian structures, $g$-natural paracontact metric structures satisfying $\widetilde h^2=0\neq \widetilde h$ and $g$-natural paracontact $(\kappa,\mu)$-spaces, as classified in Theorems~\emph{\ref{th-paracontact}, \ref{h2=0}, \ref{th-kappamu}}, provide examples of (locally) homogeneous paracontact metric manifolds of arbitrary odd dimension.
\end{Cor}


As recently proved by the first author and D. Perrone \cite{CP3}, a paracontact metric manifold $(M,\varphi,\xi,\eta,g)$ is \emph{$H$-paracontact} (that is, its characteristic vector field $\xi$ is harmonic) if and only if $\xi$ is a Ricci eigenvector. In particular, this characterisation implies that $K$-paracontact and paracontact $(\kappa,\mu)$-spaces are $H$-paracontact.

Therefore, \emph{the $g$-natural paracontact metric structures studied in Theorems}~\ref{th-parasasakian}, \ref{th-kappamu} \emph{give some large classes of examples of $H$-paracontact metric manifolds}. We will come back in a forthcoming paper to the study of the harmonicity properties of $g$-natural paracontact metric structures.



\begin{thebibliography}{99}
\bibitem{AC1}
K.~M.~T. Abbassi and G. Calvaruso, \emph{$g$-natural contact metrics on unit
tangent sphere bundles}, {Monatsh. Math.}, \textbf{151} (2006), 89--109.


\bibitem{AC2}
K.~M.~T. Abbassi and G. Calvaruso, \emph{The curvature tensor of $g$-natural metrics on unit tangent
sphere bundles}, {Int. J. Contemp. Math. Sci.}, \textbf{3}, (2008), no. 6, 245--258.



\bibitem{AK2} K.~M.~T. Abbassi and O.~Kowalski,
\emph{Naturality of homogeneous metrics on Stiefel manifolds $SO(m +1)/SO(m −1)$}, {Differential Geom. Appl.}, \textbf{28} (2010), no. 2, 131--139.


\bibitem{AS2}
K.~M.~T. Abbassi and M. Sarih,
\emph{On some hereditary properties of Riemannian $g$-natural metrics on tangent bundles of Riemannian manifolds}, Differential Geom. Appl., \textbf{22} (2005), 19--47.

\bibitem{blair-book}
D.~E. Blair, \emph{Riemannian geometry of contact and symplectic manifolds}, Second Edition. Progress in Mathematics \textbf{203}, Birkh\"{a}user, Boston, 2010.

\bibitem{BKP}
D.~E. Blair, T. Koufogiorgos and B.~J. Papantoniou,
\emph{Contact metric manifolds satisfyng a nullity condition}, {Israel J. Math.}, \textbf{91}(1995),
189--214.

\bibitem{BoPV}
E. Boeckx, D. Perrone and L. Vanhecke, \emph{Unit tangent sphere bundles and two-point homogeneous spaces},
{Period. Math. Hungar.}, \textbf{36} (1998), 79--95.


\bibitem{C1}
G. Calvaruso, \emph{Homogeneous paracontact metric three-manifolds}, Illinois J. Math., \textbf{55} (2011), 697--718.

\bibitem{CP}
G. Calvaruso and D. Perrone, \emph{Homogeneous and $H$-contact unit tangent sphere bundles},
{J. Austral. Math. Soc.}, \textbf{88} (2010),  323--337.

\bibitem{CP1}
G. Calvaruso and D. Perrone, \emph{Geometry of Kaluza-Klein metrics on the sphere $\mathbb S^3$},
{Ann. Mat. Pura Appl.}, to appear. DOI: 10.1007/s10231-012-0250-5.

\bibitem{CP2}
G. Calvaruso and D. Perrone, \emph{Metrics of Kaluza-Klein type on the anti-de Sitter space $\mathbb H_1^3$}, {Math. Nachr.}, to appear.

\bibitem{CP3}
G. Calvaruso and D. Perrone, \emph{Geometry of $H$-paracontact metric  manifolds}, submitted.

\bibitem{C}
B. Cappelletti Montano, \emph{Bi-paracontact structures and Legendre foliations},
{Kodai Math. J.} \textbf{33} (2010), 473--512.

\bibitem{CKM}
B. Cappelletti Montano, I. K\"{u}peli Erken, C. Murathan,
\emph{Nullity conditions in paracontact geometry}, {Differential Geom. Appl.,} \textbf{30} (2012) 665--693.


\bibitem{Ch}
Q.~S. Chi, \emph{A curvature characterization of certain locally rank-one symmetric spaces},
{J. Differerential Geom.}, \textbf{28} (1988), 187--202.

\bibitem{IVZ}
S. Ivanov, D. Vassilev and S. Zamkovoy, \emph{Conformal paracontact curvature and the
local flatness theorem}, Geom. Dedicata, \textbf{144} (2010), 79--100.

\bibitem{kaneyuki}
S. Kaneyuki and F.~L. Williams, \emph{Almost paracontact and parahodge structures on manifolds},
{Nagoya Math. J.}, \textbf{99} (1985), 173--187.




\bibitem{kouf97} T. Koufogiorgos,
\emph{Contact Riemannian manifolds with constant $\phi$-sectional curvature}, {Tokyo J. Math.}, \textbf{20} (1997), 55--67.

\bibitem{KSe}
O. Kowalski and M. Sekizawa, \emph{Natural transformations of Riemannian metrics on manifolds to metrics on tangent bundles-a classification}, Bull. Tokyo Gakugei Univ. (4) \textbf{40} (1988), 1--29.

\bibitem{KS}
O. Kowalski and M. Sekizawa, \emph{Invariance of $g$-natural metrics on tangent bundles}, Differential  Geometry and its Applications, 171–-181, World Sci. Publ., Hackensack, NJ, 2008.


\bibitem{N1}
Y. Nikolayevsky, \emph{Osserman conjecture in dimension $n \neq 8,16$}, {Math. Ann.}, \textbf{331} (2005), 505--522.

\bibitem{N2}
Y. Nikolayevsky, \emph{Osserman manifolds of dimension $8$},
{Manuscripta Math.}, \textbf{115} (2004), 31--53.

\bibitem{N3}
Y. Nikolayevsky, \emph{On Osserman conjecture in dimension $16$},
{Proceedings of the ``Conference on contemporary geometry and related topics''}, Belgrade, June 27--July 1, 2005.


\bibitem{zamkovoy}
S. Zamkovoy, \emph{Canonical connections on paracontact manifolds},
{Ann. Global Anal. Geom.}, \textbf{36} (2009), 37--60.

\bibitem{zamkovoy-arxiv}
S. Zamkovoy, \emph{ParaSasakian manifolds with constant paraholomorphic sectional curvature}, arXiv:0812.1676v2.
\end{thebibliography}
\end{document}